\documentclass[reqno]{amsart}

\usepackage{color}
\usepackage{amsmath}
\usepackage{amsthm}
\usepackage{graphicx}
\usepackage{amsfonts}
\usepackage{amssymb}
\usepackage{enumerate}
\usepackage{mathtools}
\newtheorem{theorem}{Theorem}[section]
\newtheorem{definition}{Definition}[section]
\newtheorem{remark}{Remark}[section]
\newtheorem{proposition}{Proposition}[section]
\newtheorem{lemma}{Lemma}[section]

\def\t{\tilde}

\def\u{\underline}

\newcommand\EE {\mathbb E}
\newcommand\FF {\mathbb F}
\newcommand\HH {\mathbb H}

\newcommand\RR {\mathbb R}

\newcommand\PP {\mathbb P}

\newcommand\cB{\mathcal B}

\newcommand\cF {\mathcal F}
\newcommand\cG {\mathcal G}
\newcommand\cH{\mathcal H}

\newcommand\cO{\mathcal O}
\newcommand\cP{\mathcal P}

\begin{document}

\title{A Probabilistic Approach to Mean Field Games with Major and Minor Players}
\author{Ren\'e Carmona  \and Xiuneng Zhu}
\maketitle

\begin{abstract}
We propose a new approach to mean field games with major and minor players. Our formulation involves a two player game
where the optimization of the representative minor player is standard while the major player faces an optimization over conditional McKean-Vlasov stochastic differential equations. The definition of this limiting game is justified by proving that its solution provides approximate Nash equilibriums for large finite player games. This proof depends upon the generalization of standard  results on the propagation of chaos to conditional dynamics. Because it is on independent interest, we prove this generalization in full detail. Using a conditional form of the Pontryagin stochastic maximum principle (proven in the appendix), we reduce the solution of the mean field game to a forward-backward system of stochastic differential equations of the conditional McKean-Vlasov type, which we solve in the Linear Quadratic setting. We use this class of models to show that Nash equilibriums in our formulation can be different from those of the formulations contemplated so far in the literature.   
\end{abstract}

\section{Introduction}
\label{se:introduction}
Stochastic  games are widely used in economic, engineering and social science applications, and the notion of Nash equilibrium is one of the most prevalent notion of equilibrium used in their analyses. However, when the number of players is large, exact Nash equilibria  are notoriously difficult to identify and construct explicitly. In an attempt to circumvent this roadblock, Lasry and Lions in \cite{MFG1,MFG2,MFG3} initiated the theory of mean field games for a type of games in which all the players are \emph{statistically identical}, and only interact through their empirical distributions. These authors  successfully identify the limiting problem as a set of two coupled PDEs, the first one of Hamilton-Jacobi-Bellman type and the second one of Kolmogorov type. Approximate Nash equilibria for the finite-player games are then derived from the solutions of the limiting problem. Motivated by the analysis of large communication networks,  Huang, Malham\'e and Caines developed independently a very similar program, see \cite{HuangMalhameCaines}, under the name of Nash Certainty Equivalence. A probabilistic approach was developed by Carmona and Delarue, see \cite{CarmonaDelarue_sicon}, in which the limiting system of coupled PDEs is replaced by a fully coupled forward-backward stochastic differential equation (FBSDE for short). Recently, an approach based on the weak formulation of stochastic controls was introduced in \cite{CarmonaLacker1} and models with a common noise studied in \cite{CarmonaDelarueLacker}..

From a modeling perspective, one of the major shortcomings of the standard mean field game theory is the strong symmetry requirement that all the players in the game are statistically identical. See nevertheless \cite{HuangMalhameCaines} where the asymptotic theory is applied to several groups of players.
The second requirement of the mean field games theory is that, when the number of players is large, the influence of one single player on the system becomes asymptotically negligible. This is in sharp contrast with some real-world applications. For example, in the banking system there are a few \emph{too big to fail}  banks, and a large number of small banks whose actions and status impact the system no-matter how large the number of small banks.

In \cite{Huang}, Huang introduced a linear-quadratic infinite-horizon model in which there exists a major player, whose influence will not fade away when the number of players tends to infinity. \cite{NguyenHuang1} introduces the finite-horizon counterpart, and \cite{NourianCaines} generalizes this model to the nonlinear case. These models are usually called '\emph{mean field game with major and minor players'}. Unfortunately, the scheme proposed in \cite{NguyenHuang1,NourianCaines} fails to accommodate the case where the state of the major player enters the dynamics of the minor players. To be more specific, in \cite{NguyenHuang1,NourianCaines}, the major player influences the minor players solely via their cost functionals. \cite{NguyenHuang2}  proposes a new scheme to solve the general case for linear-quadratic-Guassian (LQG for short) games  in which the major player's state enters the dynamics of the minor players. The limiting control problem for the major player is solved by what the authors call ``anticipative variational calculation''. In \cite{BensoussanChauYam}, the authors take, like in \cite{NourianCaines},  a stochastic Hamilton-Jacobi-Bellman approach to a type of general mean field games with major and minor players, and the limiting problem is characterized by a set of stochastic PDEs.

In this paper, we analyze a type of general mean field games with major and minor players, and develop a systematic scheme to find approximate Nash equilibria for the finite-player games using a purely probabilistic approach.  The limiting problem is identified as a two-player stochastic differential game, in which the control problem faced by the major player is of conditional McKean-Vlasov type, while the optimization problem faced by the representative minor player is a standard control problem. A matching procedure then follows the solution of the two-player game, which gives a FBSDE of McKean-Vlasov type as a characterization of the solution of the limiting problem. The construction of approximate Nash equilibria for the finite-player games with the aid of the limiting problem is also elaborated, with the approximate Nash equilibrium property carefully proved both for the major player and minor players, which fully justifies the scheme we propose. We believe that the results in this paper lead to a much more comprehensive understanding of this type of problems.

While \cite{BensoussanChauYam} is clearly the closest contribution to ours, our paper differs from \cite{BensoussanChauYam} in the following ways: first, we use a probabilistic approach based on a new version of the Pontryagin stochastic maximum principle for conditional McKean-Vlasov dynamics in order to solve the embedded stochastic control problems, while in \cite{BensoussanChauYam} a HJB equation approach is taken. Second, the limiting problem is defined as a two-player game as opposed to the three problems articulated in \cite{BensoussanChauYam}. We believe that this gives a better insight into this kind of mean field games with a major player. Third, the finite-player game in \cite{BensoussanChauYam} is a $N$-player game including only the minor players, and the major player is considered exogenous, and doesn't provide an active participation in the game. The associated propagation of chaos is then just a randomized version of the usual propagation of chaos associated to the usual mean field games, and the limiting scheme is not completely justified. Here we define the finite-player game as an $(N+1)$-player game including the major player. The construction of approximate Nash equilibriums is proved for the minor players and most importantly, for the major player as well, fully justifying our limiting scheme for finding approximate Nash equilibria.  

The classical theory of propagation of chaos, in which particles are identical is well developed. See for example the elegant treatment in \cite{Sznitman} and a more recent account in \cite{JourdainMeleardWoyczynski}. However, when introducing a major particle in the system, even when the number of particles tends to infinity, the influence of this \emph{major} particle on the other particles does not average out in the limit. This  creates interesting novel features not present in the classical theory. They involve conditioning with respect to the information flow associated to the major particle.
Our propagation of chaos result for SDEs of McKean-Vlasov type with conditional distributions is given in the stand alone Section \ref{se:conditional_chaos}.
The results of this section play a crucial role in the construction of approximate Nash equilibriums for the limiting two-player game in Section \ref{se:approximate}. They are independent of the results on Mean Field Games.  For this reason, we include them at the end of the paper, not to disrupt the flow.
 
The advantages of using the probabilistic approach are threefold. First, the probabilistic framework is natural when dealing with open-loop controls. In the present situation, the persistence of the influence of the major player forces the controls to be random, at least partially, even when looking for strategies in closed loop form. Second, the limiting conditional McKean-Vlasov control problem faced by the major player can be treated most elegantly using an appropriate version of the Pontryagin stochastic maximum principle. Since such a form of the stochastic maximum principle is not available in the published literature, we provide it in an appendix at the end of the paper. Third, our approach can rely on existing results in the literature on the well-posedness of FBSDEs and their associated decoupling fields in order to address the solvability of the limiting problem.

The mean field game model with major and minor players investigated in this paper is as follows. The major player which is indexed by $0$, can choose a control process $u^{0,N}$ taking values in a convex set $U_0 \subset \mathbb{R}^{k_0}$, and every minor player indexed by $i\in\{1,\cdots,N\}$ can choose a control process $u^{i,N}$ taking values in a convex set $U \subset \mathbb{R}^k$. The state of the system at time $t$ is given by a vector $X^N_t=(X^{0,N}_t,X^{1,N}_t,\cdots,X^{N,N}_t)\in \RR^{d_0+Nd}$ whose controlled dynamics are given by
\begin{equation}\label{fo:dynamics}
\begin{cases}
dX^{0,N}_t=b_0(t,X^{0,N}_t,\mu^N_t,u^{0,N}_t)dt+\sigma_0(t,X^{0,N}_t,\mu^N_t, u^{0,N}_t) dW^0_t,\\
dX^{i,N}_t=b(t,X^{i,N}_t,\mu^N_t,X^{0,N}_t,u^{i,N}_t)dt+\sigma(t,X^{i,N}_t,\mu^N_t,X^{0,N}_t,u^{i,N}_t) dW^i_t,\quad 1 \leq i \leq N,
\end{cases}
\end{equation}
where $(W^i_t)_{i \geq 0}$ is a sequence of independent Wiener processes, and
\begin{equation}
\label{fo:muN}
\mu^N_t=\frac{1}{N}\sum^N_{i=1}\delta_{X^{i,N}_t}
\end{equation}
is the empirical distribution of the states of the minor players, $\delta_x$ standing for the point Dirac mass at $x$.
The Wiener process $W^0$ is assumed to be $m_0$ dimensional while all the other Wiener processes $W^i$ for $i\ge 1$ are assumed to be $m$-dimensional. $X^{0,N}_t$ (and hence $b_0$) is $d_0$-dimensional while all the other $X^{i,N}_t$ (and hence $b$) are $d$-dimensional. Finally, for consistency reasons, the matrices $\sigma_0$ and $\sigma$ are $d_0\times m_0$ and $d\times m$ dimensional respectively. 
The major player aims at minimizing the cost functional given by
\begin{equation}\label{fo:majorcost}
J^{0,N}(u^{0,N},u^N)=\mathbb{E}\left[\int^T_0 f_0(t,X^{0,N}_t,\mu^N_t,u^{0,N}_t)dt+g_0(X^{0,N}_T,\mu^N_T)\right],
\end{equation}
and the minor players aim at minimizing the cost functionals:
\begin{equation}\label{fo:minorcost}
J^{i,N}(u^{0,N},u^N)=\mathbb{E}\left[\int^T_0 f(t,X^{i,N}_t,\mu^N_t,X^{0,N}_t,u^{i,N}_t)dt+g(X^{i,N}_T,\mu^N_T,X^{0,N}_T)\right], \quad 1 \leq i \leq N.
\end{equation}
We use the notation $u^N$ for $(u^{1,N},\cdots,u^{N,N})$. We observe readily that an important difference between the current model and the usual mean field game model is the presence of the state of the major player in the state dynamics and the cost functionals of the minor players. Even when the number of minor players is large, the major player can still influence the behavior of the system in a non-negligible manner.

The rest of the paper is organized as follows. In the preliminary section \ref{se:prelim} we review briefly the usual mean field game scheme, and then proceed to the scheme for the mean field games with major and minor players proposed in this paper. Some heuristic arguments leading to the scheme are also provided, and the difference between the current scheme and the one used in \cite{NguyenHuang1,NourianCaines} are highlighted. In section \ref{se:mfg} we carry out the scheme described in section \ref{se:prelim} for a type of mean field games with major and minor players with scalar interactions, and we use the Pontryagin maximum principle to solve the embedded stochastic control problems. The FBSDE of conditional mean field type characterizing the Nash equilibria for the limiting two-player game is derived. In section \ref{se:approximate}, we prove that the solution of the limiting problem can actually be used to build approximate Nash equilibria for the finite-player games, justifying our scheme. In section \ref{se:lqg}, we apply the scheme to the case of Linear Quadratic Gaussian (LQG for whort) models, and find explicit approximate Nash equilibria for the finite-player games, and in section \ref{se:example} a concrete example is given to show that the current scheme leads to different results from the scheme proposed in \cite{NguyenHuang2} and \cite{NourianCaines}. In  the independent section \ref{se:conditional_chaos}, we prove a conditional version of propagation of chaos which plays a pivotal role in the construction of approximate Nash equilibria in section \ref{se:approximate}. Finally, in the appendix at the end of the paper, we prove a version of the sufficient part of the Pontryagin stochastic maximum principle for conditional McKean-Vlasov dynamics used in solving the stochastic control problem faced by the major player.

\section{Preliminaries}
\label{se:prelim}
\subsection{Brief Review of the Standard Mean Field Game Problem}
A standard introduction to the mean field game (MFG for short) theory starts with an $N$-player stochastic differential game, the dynamics of the states of the players being governed by stochastic differential equations (SDEs)
$$
dX^{i,N}_t=b(t,X^{i,N}_t, \mu^N_t,u_t^{i,N})dt+\sigma(t,X^{i,N}_t,\mu^N_t,u^{i,N}_t) dW_t^i, \quad i=1,2,...,N,
$$
each player aiming at the minimization of a cost functional
$$
J^{i,N}(u)=\mathbb{E}\left[\int_0^T f(t,X_t^{i,N},\mu^N_t,u_t^{i,N})dt+g(X_T^{i,N},\mu^N_T)\right],
$$
where $\mu^N_t$ stands for the empirical distribution of the $X^{N,i}_t$ for $i=1,\cdots,N$. The usual MFG scheme can be summarized in the following 3 steps:
\begin{enumerate}
\item Fix a deterministic flow $(\mu_t)_{0 \leq t \leq T}$ of probability measures.
\item Solve the standard stochastic control problem: minimize
$$J(u)=\mathbb{E}\left[\int_0^Tf(t,X_t, \mu_t,u_t)dt+g(X_T,\mu_T)\right],$$
when the controlled dynamics of the process $X_t$ are given by
$$dX_t=b(t,X_t, \mu_t, u_t)dt+\sigma(t,X_t,\mu_t,u_t) dW_t.$$
\item  Solve the fixed point problem $\Phi(\mu)=\mu$, where for each flow $\mu$ as in step (1), $\Phi(\mu)$ denotes the flow of marginal distributions of the optimally controlled state process found in step (2).
\end{enumerate}
If the above scheme can be carried out successfully, it is usually possible to prove that the optimal control found in step (2) can be used to provide approximate Nash equilibriums for the finite-player game. The interested reader is referred to \cite{MFG1,MFG2,MFG3,HuangMalhameCaines} for detailed discussions of the PDE approach of the above scheme and to \cite{CarmonaDelarue_sicon, CarmonaLacker1} for two different probabilistic approaches.

\subsection{Heuristic derivation of MFG approach}
In this subsection we provide a heuristic argument which leads to a scheme for mean field games with major and minor players. The finite-player games are described by equations (\ref{fo:dynamics})-(\ref{fo:minorcost}) above. Because all the minor players are identical and influenced by the major player in exactly the same way, it is reasonable to assume that they are exchangeable, even when the optimal strategies (in the sense of Nash equilibrium) are implemented. On the other hand, for any sequence of integrable exchangeable random variables $(X_i)_{i \geq 1}$, de Finetti's law of large numbers states that almost surely,
$$
\frac{1}{N}\sum^N_{i=1} \delta_{X_i} \Longrightarrow \mathcal{L}(X_1 \vert \mathcal{G}),
$$
for some $\sigma$-field $\mathcal{G}$ where $ \Longrightarrow$ denotes convergence in distribution. We may want to apply this result for each time $t$ to the individual states $X^{i,N}_t$ in which case, a natural candidate for the $\sigma$-field $\mathcal{G}$ could be the element $\mathcal{F}^0_t$ of the filtration  generated by the Wiener process $W^0$ driving the dynamics of the state of the major player. This suggests that in mean field games with major and minor players, we can proceed essentially in the same way as in the standard mean field game theory, except for the fact that instead of fixing a \emph{deterministic} measure flow in the first step, we fix an adapted \emph{stochastic} measure flow, and in the last step, match this stochastic measure flow to the flow of marginal conditional distribution of the state of the representative minor player given $\mathcal{F}^0_t$. This is in accordance with intuition since, as all the minor players are influenced by the major player, they should make their decisions conditioned on the information provided by the major player. Notice that this is also consistent with the procedure used in the presence of a so-called common noise as
investigated in \cite{CarmonaDelarueLacker}.

However, the above argument fails to apply to the major player. Indeed, no matter how many minor players are present in the game, the major player's control influences all the minor players, and in particular, the empirical distribution formed by the minor players. When we construct the limiting problem for the major player, it is thus more reasonable to allow the major player to control the stochastic measure flow, instead of fixing it \emph{a priori}. This asymmetry between major and minor players was also observed in \cite{BensoussanChauYam}. 

\subsection{Precise formulation of the MFG problem with major and minor players}
Using the above heuristic argument, we arrive at the following scheme for the major-minor mean field game problem. The limiting control problem for the major player is of conditional McKean-Vlasov type, where the measure flow is endogenous, and the limiting control problem for the representative minor player is a standard one, where the measure flow is exogenous and fixed at the beginning of the scheme. As a consequence, the limiting problem becomes a two-player stochastic differential game between the major player and a representative minor player, instead of two consecutive stochastic control problems for each of them. Specifically:

\begin{enumerate}
\item Fix a $\FF^0$-progressively measurable stochastic measure flow $(\mu_t)_{0 \leq t \leq T}$ where $\FF^0=(\cF^0_t)_{t\ge 0}$ denotes the filtration generated by the Wiener process $W^0$.
\item Consider the following two-player stochastic differential game where  the control $(u^0_t)_{0 \leq t \leq T}$ of the first player is assumed to be adapted to $\FF^0$, and the control $(u_t)_{0 \leq t \leq T}$ of the second player is assumed to be adapted to the filtration $\FF=(\cF_t)_{t\ge 0}$ generated by $W$, and where the controlled dynamics of the state of the system are given by

\begin{equation}
\label{fo:limitSDE}
\begin{cases}
dX^0_t=b_0(t,X^0_t,\mathcal{L}(X_t \vert \mathcal{F}^0_t),u^0_t)dt+\sigma_0 (t,X^0_t,\mathcal{L}(X_t\vert\mathcal{F}^0_t),u^0_t)dW^0_t,\\
dX_t=b(t,X_t,\mathcal{L}(X_t\vert \mathcal{F}^0_t),X^0_t,u_t)dt+\sigma(t,X_t,\mathcal{L}(X_t\vert \mathcal{F}^0_t),X^0_t,u_t) dW_t,\\
d\check{X}^0_t=b_0(t,\check{X}^0_t,\mu_t,u^0_t)dt+\sigma_0(t,\check{X}^0_t,\mu_t,u^0_t) dW^0_t,\\
d\check{X}_t=b(t,\check{X}_t,\mu_t,\check{X}^0_t,u_t)dt+\sigma(t,\check{X}_t,\mu_t,\check{X}^0_t,u_t) dW_t,
\end{cases}
\end{equation}
and the cost functionals for the two players are given by
$$\begin{aligned}
&J^0(u^0,u)=\mathbb{E}\left[\int^T_0 f_0(t,X^0_t,\mathcal{L}(X_t \vert \mathcal{F}^0_t),u^0_t)dt+g_0(X^0_T,\mathcal{L}(X_T\vert \mathcal{F}^0_T))\right],\\
&J(u^0,u)=\mathbb{E}\left[\int^T_0 f(t,\check{X}_t,\mu_t,\check{X}^0_t,u_t)dt+g(\check{X}_T,\mu_T,\check{X}^0_T)\right],
\end{aligned}$$
where $\mathcal{L}(X_t \vert \mathcal{F}^0_t)$ stands for the conditional distribution of $X_t$ given $\mathcal{F}^0_t$. We look for Nash equilibria for this game.
\item Satisfy the consistency condition
\begin{equation}\label{fo:consistency}
\mu_t=\mathcal{L}(X_t \vert \mathcal{F}^0_t), \quad \forall t \in [0,T],
\end{equation}
where $X_t$ is the second component of the state controlled by $u^0$ and $u$ giving the Nash equilibrium found in step (2). 
\end{enumerate}
Notice that the above consistency condition amounts to solving a fixed point problem in the space of stochastic measure flows.
Notice also that even when the $X^{i,N}_t$ are scalar, the system \eqref{fo:limitSDE} describes the dynamics of a $4$-dimensional state driven by two independent Wiener processes. The dynamics of the first two components are of the conditional McKean-Vlasov type  (because of the presence of the conditional distribution $\mathcal{L}(X_t\vert \mathcal{F}^0_t)$ of $X_t$ in the coefficients) while the dynamics of the last two components are given by standard stochastic differential equations with random coefficients. In this two player game, the cost functional $J^0$ of the major player is of the McKean-Vlasov type while the cost functional $J$ of the representative minor player is of the standard type. As explained earlier, this is the main feature of our formulation of the problem.
 
Later in the paper, we show that if we are able to find a fixed point in the third step, i.e. a stochastic measure flow $(\mu_t)_{0 \leq t \leq T}$ satisfying (\ref{fo:consistency}), we can use it to construct approximate Nash equilibria for the finite-player games when the number of players is sufficiently large.  The precise meaning of this statement will be made clear in section \ref{se:approximate}.

\section{Mean Field Games with Major and Minor Players: The General Case}
\label{se:mfg}
In this section we  analyze in detail the scheme explained in the previous section, and we derive a FBSDE characterizing the solution to the limiting problem. 
We assume that $\Omega$ is a standard space and $\cF$ is its Borel $\sigma$-field, so that regular conditional distributions exist for all sub-$\sigma$-fields. The definition of standard probability spaces we use here can be found in \cite{Cardaliaguet}.
The finite-player games are described by (\ref{fo:dynamics})-(\ref{fo:minorcost}) where $(W^i)_{i \geq 0}$ is a sequence of independent Wiener processes. We shall use the following assumptions.

\noindent (\textbf{A1}) There exists a constant $c>0$ such that for all $t \in [0,T]$, $x'_0,x_0 \in \mathbb{R}^{d_0}$, $x',x \in \mathbb{R}^d$, $\mu',\mu \in \mathcal{P}_2(\mathbb{R}^d)$, $u_0 \in U_0$ and $u \in U$ we have
\begin{equation}
\begin{aligned}
&\vert (b_0,\sigma_0)(t,x'_0,\mu',u'_0)-(b_0,\sigma_0)(t,x_0,\mu,u_0)\vert+\vert (b,\sigma)(t,x',\mu',x'_0,u')-(b,\sigma)(t,x,\mu,x_0,u)\vert \\
&\phantom{????????}\leq c\bigg(\vert x'_0-x_0\vert+\vert x'-x\vert+\vert u'_0-u_0\vert+\vert u'-u\vert+W_2(\mu',\mu)\bigg).
\end{aligned}
\end{equation}
\noindent (\textbf{A2}) For all $u_0 \in U_0$ and $u \in U$ we have
$$
\mathbb{E}\left[\int^T_0 \vert (b_0,\sigma_0)(t,0,\delta_0,u_0)\vert^2+\vert (b,\sigma)(t,0,\delta_0,0,u)\vert^2\right]< \infty.
$$
\noindent (\textbf{A3}) There exists a constant $c_L>0$ such that for all $x_0,x'_0 \in \mathbb{R}^{d_0}$, $u_0,u'_0 \in \mathbb{R}^{k_0}$ and $\mu,\mu' \in \mathcal{P}_2(\mathbb{R}^d)$, we have
$$
\begin{aligned}
&\vert (f_0,g_0)(t,x'_0,\mu',u'_0)-(f_0,g_0)(t,x_0,\mu,u_0)\vert\\
 &\phantom{????????} \leq c_L\bigg(1+\vert (x'_0,u'_0)\vert+\vert (x_0,u_0)\vert+M_2(\mu')+M_2(\mu)\bigg)\bigg(\vert (x'_0,u'_0)-(x_0,u_0)\vert+W_2(\mu',\mu)\bigg),
\end{aligned}
$$
and for all $x_0 \in \mathbb{R}^{d_0}$, $x,x'\in \mathbb{R}^{d}$, $u,u' \in \mathbb{R}^{k}$ and $\mu,\mu' \in \mathcal{P}_2(\mathbb{R}^d)$,
$$
\begin{aligned}
&\vert (f,g)(t,x',\mu',x_0,u')-(f,g)(t,x,\mu,x_0,u)\vert\\
&\phantom{????????}\leq c_L\bigg(1+\vert (x',u')\vert+\vert (x,u)\vert+M_2(\mu')+M_2(\mu)\bigg)\bigg(\vert (x',u')-(x,u)\vert+W_2(\mu,\mu')\bigg).
\end{aligned}
$$
where $\mathcal{P}_2(\mathbb{R}^d)$ denotes the set of probability measures of order $2$ (i.e. with a finite second moment), and $W_2(\mu,\mu')$ the $2$-Wasserstein distance between $\mu,\mu'\in\cP_2(\RR^d)$. Also, we used the notation $M_2(\mu)=\int |x|^2 \mu(dx)$ for the second moment of $\mu$.
\vskip 2pt\noindent 
(\textbf{A4}) The functions $b_0$, $b$, $f$ and $g$ are differentiable in $x_0$, $x$ and $\mu$. Differentiability with respect to measure arguments is discussed in the appendix at the end of the paper.

Assumptions (A1)-(A2) guarantee that for all admissible controls, the SDEs (\ref{fo:dynamics})-(\ref{fo:minorcost}) and (\ref{fo:limitSDE}) have unique solutions, and (A3) guarantees that the associated cost functionals are well-defined. Assumption (A4) will be used when we define adjoint processes for the limiting control problems. 

In the following, we use $\mathbb{S}^{2,d}(\FF;U)$ to denote all $\FF$-progressively measurable processes $X$ taking values in $U \subset \mathbb{R}^d$ such that
\begin{equation}
\label{fo:S2d}
\mathbb{E}\left[\sup_{0\leq t \leq T} | X_t|^2\right]< \infty,
\end{equation}
$\mathbb{H}^{2,d}(\FF;U)$ to denote all $U$-valued $\FF$-progressively measurable processes $X$ such that
\begin{equation}
\label{fo:H2d}
\mathbb{E}\left[\int^T_0 |X_t|^2\right]< \infty,
\end{equation}
and finally we use $\mathcal{M}^{2,d}(\FF)$ to denote the set of $\FF$-progressively measurable stochastic measure flows $\mu$ on $\mathbb{R}^d$ such that
\begin{equation}
\label{fo:M2d}
\mathbb{E}\left[\int^T_0 \int_{\mathbb{R}^d}| x|^2 d\mu_t\right]<\infty.
\end{equation}
We will omit the filtration $\FF$ and the domain $U$ when there is no risk of confusion.

\subsection{Control problem for the major player}
In this subsection we consider the limiting two-player game and search for the major player's best response $u^0$ to the control $u$ of the representative minor player. This amounts to solving the optimal control problem based on the controlled dynamics
\begin{equation}
\begin{cases}
dX^0_t=b_0(t,X^0_t,\mathcal{L}(X_t \vert \mathcal{F}^{0}_t),u^0_t)dt+\sigma_0(t,X^0_t,\mathcal{L}(X_t\vert\mathcal{F}^0_t),u^0_t) dW^0_t,\quad X^0_0=x^0_0,\\
dX_t=b(t,X_t,\mathcal{L}(X_t \vert \mathcal{F}^{0}_t),X^0_t,u_t)dt+\sigma(t,X_t,\mathcal{L}(X_t \vert \mathcal{F}^{0}_t),X^0_t,u_t) dW_t,\quad X_0=x_0,
\end{cases}
\end{equation}
and the cost functional
$$
J^0(u^0)=\mathbb{E}\left[\int^T_0 f_0(t,X^0_t,\mathcal{L}(X_t \vert \mathcal{F}^{0}_t),u^0_t)dt+g_0(X^0_T,\mathcal{L}(X_T\vert \mathcal{F}^0_T))\right],
$$
where it is assumed that the control $u$ is given, the set of admissible controls $u^0$ being the space $\mathbb{H}^{2,k_0}(\FF^0;U_0)$. In what follows, this stochastic control problem will be denoted by (P1). We check readily that conditions (A2.1) - (A2.3) in the appendix at the end of the paper are satisfied. The Hamiltonian is defined as
\begin{equation}
\label{fo:H0}
\begin{aligned}
&H_0(t,x_0,x,\mu,p_0,p,q_{00},q_{11},u_0,u)=\langle p_0, b_0(t,x_0,\mu,u_0)\rangle+\langle p, b(t,x,\mu,x_0,u)\rangle\\
                             &\phantom{???????????????}+\langle q_{00},\sigma_0(t,x_0,\mu,u_0)\rangle+\langle q_{11},\sigma(t,x,\mu,x_0,u)\rangle+ f_0(t,x_0,\mu,u_0).
\end{aligned}
\end{equation}
We then introduce the following assumption regarding minimization of this Hamiltonian.

\noindent (\textbf{M0})  For all fixed $(t,x_0,x,\mu,p_0,p,q_{00},q_{11},u)$ there exists a unique minimizer of the Hamiltonian $H_0$ as a function of $u_0$. Note that this minimizer should not depend upon $p$, $q_{11}$ and $u$. It will be denoted by $\hat{u}^0(t,x_0,\mu,p_0,q_{00})$.

\begin{remark}
This assumption is satisfied when the running cost $f_0$ is strictly convex in $u^0$, the drift $b_0$ is linear in $u^0$ and the volatility $\sigma_0$ is uncontrolled in the sense that it does not depend upon $u^0$. This will be the case in the examples considered later on. 
\end{remark}

For each admissible control $u^0$, the associated adjoint process $(P^0,P,Q^{00}, Q^{01}, Q^{10}, Q^{11})$ is defined as the solution of the backward stochastic differential equation (BSDE):
\begin{equation}\label{fo:BSDEmajor}
\begin{cases}
dP^0_t=-\partial_{x_0} H_0(t,\underline{X}_t,\mathcal{L}(X_t\vert\mathcal{F}^0_t),\underline{P}_t,\underline{Q}_t,u^0_t,u_t)dt+Q^{00}_t dW^0_t+Q^{01}_t dW_t,\\
\begin{aligned}
dP_t=&-\partial_x H_0(t,\underline{X}_t,\mathcal{L}(X_t\vert\mathcal{F}^0_t),\underline{P}_t,\underline{Q}_t,u^0_t,u_t)dt+Q^{10}_t dW^0_t+Q^{11}_t dW_t\\
&-\mathbb{E}^{\mathcal{F}^0_t}[\partial_\mu H_0(t,\tilde{\underline{X}}_t,\mathcal{L}(\tilde{X}_t\vert\mathcal{F}^0_t),\tilde{\underline{P}}_t,\tilde{\underline{Q}}_t,u^0_t,u_t)(X_t)]dt
\end{aligned},\\
P^0_T=\partial_{x_0} g(X^0_T,\mathcal{L}(X_T\vert \mathcal{F}^{0}_t)),\\
P_T=\mathbb{E}^{\mathcal{F}^0_T}[\partial_\mu g(\tilde{X}^0_T,\mathcal{L}(\tilde{X}_T\vert \mathcal{F}^0_T))(X_T)],
\end{cases}
\end{equation}
where to lighten the notations we write $\underline{X}=(X^0,X)$, $\underline{P}=(P^0,P)$ and $\underline{Q}=(Q^{00},Q^{01},Q^{10},Q^{11})$. We refer the reader to appendix at the end of the paper for 1) definitions of the tilde notation, which provides a natural extension of random variables to an extension of the original probability space, and of  $\mathbb{E}^{\mathcal{F}^0_t}[\cdot]$ which denotes expectation with respect to the regular conditional distribution on an extension of the original probability space, and 2) references to the definition and the properties of the differentiation with respect to the measure argument. Despite the presence of the conditional distributions in the coefficients, standard proofs of existence and uniqueness of solutions of BSDEs with Lipschitz coefficients still apply to  (\ref{fo:BSDEmajor}), for example when the derivatives of assumption (A4) are uniformy Lipshitz with linear growth. See for example \cite{CarmonaDelarue_ecp}.

In order to minimize the complexity of the notation, we systematically add a bar on the top of a random variable to denote its conditional expectation with respect to $\mathcal{F}^0_t$, for example $\bar{P}^0$ stands for $\EE[P^0|\cF^0_t]$. 

Once properly extended to cover the present situation, (see \cite{CarmonaDelarue_ap} for the necessary condition in the unconditional case, and the appendix for the sufficient condition) the necessary part of the Pontryagin stochastic maximum principle says that, if the control $u^0=(u^0_t)_t$ is optimal, then the Hamiltonian \eqref{fo:H0} is minimized along the trajectory of $(X^0_t,X_t,\underline{P}_t,\underline{Q}_t)$. So given assumption (M0) and the sufficient condition of the stochastic maximum principle proven in the  appendix at the end of the paper, $\hat{u}^0_t=\hat{u}^0(t,X^0_t,\mathcal{L}(X_t \vert \mathcal{F}^0_t),\bar{P}^0_t,\bar{Q}^{00}_t)$ will be an optimal control for the problem at hand if we can solve the forward backward stochastic differential equation (FBSDE):

\begin{equation}
\label{fo:FBSDEM}
\begin{cases}
dX^0_t=\partial_{p_0} H_0(t,\underline{X}_t,\mathcal{L}(X_t\vert\mathcal{F}^0_t),\underline{P}_t,\underline{Q}_t,\hat{u}^0_t,u_t)dt+\partial_{q_{00}}H_0(t,\underline{X}_t,\mathcal{L}(X_t\vert\mathcal{F}^0_t),\underline{P}_t,\underline{Q}_t,\hat{u}^0_t,u_t) dW^0_t,\\
dX_t=\partial_p H_0(t,\underline{X}_t,\mathcal{L}(X_t\vert\mathcal{F}^0_t),\underline{P}_t,\underline{Q}_t,\hat{u}^0_t,u_t)dt+\partial_{q_{11}}H_0(t,\underline{X}_t,\mathcal{L}(X_t\vert\mathcal{F}^0_t),\underline{P}_t,\underline{Q}_t,\hat{u}^0_t,u_t) dW_t,\\
dP^0_t=-\partial_{x_0} H_0(t,\underline{X}_t,\mathcal{L}(X_t\vert\mathcal{F}^0_t),\underline{P}_t,\underline{Q}_t,\hat{u}^0_t,u_t)dt+Q^{00}_t dW^0_t+Q^{01}_t dW_t,\\
\begin{aligned}
dP_t=&-\partial_x H_0(t,\underline{X}_t,\mathcal{L}(X_t\vert\mathcal{F}^0_t),\underline{P}_t,\underline{Q}_t,\hat{u}^0_t,u_t)dt+Q^{10}_t dW^0_t+Q^{11}_t dW_t\\
&-\mathbb{E}^{\mathcal{F}^0_t}[\partial_\mu H_0(t,\tilde{\underline{X}}_t,\mathcal{L}(\tilde{X}_t\vert\mathcal{F}^0_t),\tilde{\underline{P}}_t,\tilde{\underline{Q}}_t,\tilde{\hat{u}}^0_t,u_t)(X_t)]dt
\end{aligned}
\end{cases}
\end{equation}
with the initial and terminal conditions given by
$$
X^0_0=x^0_0,\quad X_0=x_0,\qquad P^0_T=\partial_{x_0} g(X^0_T,\mathcal{L}(X_T\vert \mathcal{F}^{0}_t)),\qquad P_T=\mathbb{E}^{\mathcal{F}^0_T}[\partial_\mu g(\tilde{X}^0_T,\mathcal{L}(\tilde{X}_T\vert \mathcal{F}^0_T))(X_T)].
$$
In general, FBSDEs are more difficult to solve than BSDEs. This is even more apparent in the case of equations of the McKean-Vlasov type. See nevertheless \cite{CarmonaDelarue_ecp} for an existence result in the unconditional case. In its full generality, the solvability of  FBSDE \eqref{fo:FBSDEM} of conditional McKean-Vlasov type is beyond the scope of this paper. We will solve it only in the linear quadratic case.

We show in the appendix that appropriate convexity assumptions are sufficient for optimality. We summarize them for later reference.

\noindent (\textbf{C0}) The function $\mathbb{R}^{d_0} \times \mathcal{P}_2(\mathbb{R}^d) \ni (x,\mu) \hookrightarrow g(x,\mu)$ is convex.  The function
$$\mathbb{R}^{d_0} \times \mathbb{R}^d \times \mathcal{P}_2(\mathbb{R}^d) \times  U_0 \ni (x_0,x,\mu,u_0)\hookrightarrow H(t,x_0,x,\mu,p_0,p,q_{00},q_{11},u_0,u)$$
is convex for all fixed $(t,p_0,p,q_{00},q_{11},u)$.

We then have the following proposition.

\begin{proposition}
Let us assume that (A1)-(A3), (M0) and (C0) are in force. If 
$$(X^0,X,P^0,P,Q^{00}, Q^{01}, Q^{10}, Q^{11}) \in \mathbb{S}^{2,d_0+d} \times \mathbb{S}^{2,d_0+d} \times \mathbb{H}^{2, (d_0+d)\times(d_0+d)}
$$ 
is a solution to the FBSDE (\ref{fo:FBSDEM}), then $u^0_t=\hat{u}^0(t,X^0_t,\mathcal{L}(X_T\vert \mathcal{F}^0_t),\bar{P}^0_t,\bar{Q}^{00}_t),$ is an optimal control for problem (P1)  and $(X^0,X)$ is the associated optimally controlled state process.
\end{proposition}

\subsection{Control problem for the representative minor player}
For the representative minor player's best response control problem, for each fixed stochastic measure flow $\mu$ in $\mathcal{M}^{2,d}(\FF^0)$  and for each admissible control $u^0=(u^0_t)_t$ of the major player, we solve the optimal control problem of the controlled dynamics
\begin{equation}
\label{fo:minorsde}
\begin{cases}
d\check{X}^0_t=b_0(t,\check{X}^0_t,\mu_t,u^0_t)dt+\sigma_0(t,\check{X}^0_t,\mu_t,u^0_t) dW^0_t,\quad \check{X}^0_0=x^0_0,\\
d\check{X}_t=b(t,\check{X}_t,\mu_t,\check{X}^0_t,u_t)dt+\sigma(t,\check{X}_t,\mu_t,\check{X}^0_t,u_t) dW_t,\quad \check{X}_0=x_0
\end{cases}
\end{equation}
for the cost functional
\begin{equation}
J(u)=\mathbb{E}\bigg[\int^T_0 f(t,\check{X}_t,\mu_t,\check{X}^0_t,u_t)+g(\check{X}_T,\mu_T,\check{X}^0_T)\bigg].
\end{equation}
Note that since $u^0$ and $\mu$ are fixed, the first SDE in \eqref{fo:minorsde} can be solved \emph{off line}, and its solution appears in the second SDE of \eqref{fo:minorsde} and the cost functional only as an exogenous source of randomness. 
If we choose the set of admissible controls for the representative minor player to be $\mathbb{H}^{2,k}(\FF^{W_0,W};U)$ where $\FF^{W_0,W}$ is the filtration generated by both Wiener processes $W^0$ and $W$, this problem is a standard non-Markovian stochastic control problem. We shall denote it by (P2) in the following. For this reason, we introduce only adjoint variables for $\check{X}_t$, and use the reduced Hamiltonian:
\begin{equation}
\label{fo:minorH}
H(t,x_0,x,\mu,y,z_{11},u^0,u)=\langle y, b(t,x,\mu,x_0,u)\rangle\\
                                     +\langle z_{11},\sigma(t,x,\mu,x_0,u)\rangle+f(t,x,\mu,x_0,u).
\end{equation}
As before, in order to find a function satisfying the Isaacs condition, we introduce the following assumption regarding its minimization.

\noindent (\textbf{M}) For all fixed $(t,x_0,x,\mu,y,z_{11},u_0)$, there exists a unique minimizer of the above reduced Hamiltonian $H$ as a function of $u$. This minimizer will be denoted by $\hat{u}(t,x_0,x,\mu,y,z_{11})$.

For all admissible control $u$ we can define the adjoint process $(\underline{Y},\underline{Z})=(Y^0,Y,Z^{00},Z^{01},Z^{10},Z^{11})$ associated to $u$ as the solution of the following BSDE:
\begin{equation}\label{fo:BSDEminor}
\begin{cases}
dY^0_t=-\partial_{x_0}H(t,\underline{\check{X}}_t,\mu_t,\underline{Y}_t,\underline{Z}_t,u^0_t,u_t)dt+Z^{00}_t dW^0_t+Z^{01}_t dW_t,\\
dY_t=-\partial_x H(t,\underline{\check{X}}_t,\mu_t,\underline{Y}_t,\underline{Z}_t,u^0_t,u_t)dt+Z^{10}_t dW^0_t+Z^{11}_t dW_t,\\
Y^0_T=\partial_{x_0}g(\check{X}_T,\mu_T,\check{X}^0_T),\qquad
Y_T=\partial_x g(\check{X}_T,\mu_T,\check{X}^0_T).
\end{cases}
\end{equation}
The existence of the adjoint processes associated to a given admissible control $u$ is a consequence of the standard existence result of solutions of BSDEs
when the partial derivatives of $b$, $\sigma$ and $f$ with respect to $x_0$ and $x$ are uniformly bounded in $(t,x_0,x,\mu)$.
The necessary part of the Pontryagin stochastic maximum principle says that, if the admissible control $u=(u_t)_t$ is optimal, then the Hamiltonian \eqref{fo:minorH} is minimized along the trajectory of $(X^0_t,X^t,\underline{Y}_t,\underline{Z}_t)$. So given assumption (M) and the sufficient condition of the stochastic maximum principle (see for example the appendix in section \ref{se:Pontryagin}), $\hat{u}_t=\hat{u}(t,X^0_t,X_t, \mathcal{L}(X_t \vert \mathcal{F}^0_t),\u{Y}_t,\u{Z}_t)$ will be an optimal control for the problem at hand if we can solve the forward backward stochastic differential equation (FBSDE):

The standard Pontryagin maximum principle tells us that the optimal control should be given by $\hat{u}_t=\hat{u}(t,\check{X}^0_t,\check{X}_t,\mu_t,Y_t,Z^{11}_t)$, and plugging this expression into the controlled dynamics and BSDE (\ref{fo:BSDEminor}) gives us the following FBSDE:
\begin{equation}\label{fo:FBSDEminor}
\begin{cases}
d\check{X}^0_t=\partial_{y_0} H(t,\underline{\check{X}}_t,\mu_t,\underline{Y}_t,\underline{Z}_t,u^0_t,\hat{u}_t)dt+\partial_{z_{00}}H(t,\underline{\check{X}}_t,\mu_t,\underline{Y}_t,\underline{Z}_t,u^0_t,\hat{u}_t) dW^0_t,\\
d\check{X}_t=\partial_{y} H(t,\underline{\check{X}}_t,\mu_t,\underline{Y}_t,\underline{Z}_t,u^0_t,\hat{u}_t)dt+\partial_{z_{11}}H(t,\underline{\check{X}}_t,\mu_t,\underline{Y}_t,\underline{Z}_t,u^0_t,\hat{u}_t) dW_t,\\
dY^0_t=-\partial_{x_0} H(t,\underline{\check{X}}_t,\mu_t,\underline{Y}_t,\underline{Z}_t,u^0_t,\hat{u}_t)dt+Z^{00}_t dW^0_t+Z^{01}_t dW_t,\\
dY_t=-\partial_x H(t,\underline{\check{X}}_t,\mu_t,\underline{Y}_t,\underline{Z}_t,u^0_t,\hat{u}_t)dt+Z^{10}_t dW^0_t+Z^{11}_t dW_t,
\end{cases}
\end{equation}
with the initial and terminal conditions given by
$$
\check{X}^0_0=x^0_0,\quad
\check{X}_0=x_0,\qquad
Y^0_T=\partial_{x_0} g(\check{X}_T,\mu_T,\check{X}^0_T),\quad
Y_T=\partial_x g(\check{X}_T,\mu_T,\check{X}^0_T).
$$
We also need the following convexity assumption.

\noindent (\textbf{C}) The function $\mathbb{R}^d \times \mathcal{P}_2(\mathbb{R}^d) \times \mathbb{R}^{d_0} \ni (x,\mu,x_0) \hookrightarrow g(x,\mu,x_0)$ is convex in $(x_0,x)$. The function
$$\mathbb{R}^{d_0} \times \mathbb{R}^d \times \mathcal{P}_2(\mathbb{R}^d) \times U \ni (x_0,x,\mu,u) \hookrightarrow H(t,x_0,x,\mu,y_0,y,z_{00},z_{11},u_0,u)$$
is convex for all $(t,y_0,y,z_{00},z_{11},u_0)$. Then we have the following proposition.
\begin{proposition}
Assuming that (A1-2), (M) and (C) are in force, if $(\check{X}^0,\check{X},Y^0,Y,Z^{00},Z^{01},Z^{10},Z^{11}) \in \mathbb{S}^{2,d_0+d} \times \mathbb{S}^{2,d_0+d} \times \mathbb{H}^{2, (d_0+d)\times(d_0+d)}$ is a solution to the FBSDE (\ref{fo:FBSDEminor}), then an optimal control of the control problem (P2) is given by
$$u_t=\hat{u}(t,\check{X}^0_t,\check{X}_t,\mu_t,Y_t,Z^{11}_t),$$
and $(\check{X}^0,\check{X})$ is the associated optimally controlled state process.
\end{proposition}

\subsection{Nash equilibrium for the limiting two-player game}
By the very definition of Nash equilibria, the following proposition is self-explanatory.
\begin{proposition}
Assume that (A1-2), (M0), (M), (C0) and (C) are in force. Consider the following FBSDE:
\begin{equation}\label{fo:FBSDEMm}
\begin{cases}
dX^0_t=\partial_{p_0} H_0(t,\underline{X}_t,\mathcal{L}(X_t\vert\mathcal{F}^0_t),\underline{P}_t,\underline{Q}_t,\hat{u}^0_t,\hat{u}_t)dt+\partial_{q_{00}}H_0(t,\underline{X}_t,\mathcal{L}(X_t\vert\mathcal{F}^0_t),\underline{P}_t,\underline{Q}_t,\hat{u}^0_t,\hat{u}_t) dW^0_t,\\
dX_t=\partial_p H_0(t,\underline{X}_t,\mathcal{L}(X_t\vert\mathcal{F}^0_t),\underline{P}_t,\underline{Q}_t,\hat{u}^0_t,\hat{u}_t)dt+\partial_{q_{11}}H_0(t,\underline{X}_t,\mathcal{L}(X_t\vert\mathcal{F}^0_t),\underline{P}_t,\underline{Q}_t,\hat{u}^0_t,\hat{u}_t) dW_t,\\
d\check{X}^0_t=\partial_{y_0} H(t,\underline{\check{X}}_t,\mu_t,\underline{Y}_t,\underline{Z}_t,\hat{u}^0_t,\hat{u}_t)dt+\partial_{z_{00}}H(t,\underline{\check{X}}_t,\mu_t,\underline{Y}_t,\underline{Z}_t,\hat{u}^0_t,\hat{u}_t) dW^0_t,\\
d\check{X}_t=\partial_{y} H(t,\underline{\check{X}}_t,\mu_t,\underline{Y}_t,\underline{Z}_t,\hat{u}^0_t,\hat{u}_t)dt+\partial_{z_{11}}H(t,\underline{\check{X}}_t,\mu_t,\underline{Y}_t,\underline{Z}_t,\hat{u}^0_t,\hat{u}_t) dW_t,\\
dP^0_t=-\partial_{x_0} H_0(t,\underline{X}_t,\mathcal{L}(X_t\vert\mathcal{F}^0_t),\underline{P}_t,\underline{Q}_t,\hat{u}^0_t,\hat{u}_t)dt+Q^{00}_t dW^0_t+Q^{01}_t dW_t,\\
\begin{aligned}
dP_t=&-\partial_x H_0(t,\underline{X}_t,\mathcal{L}(X_t\vert\mathcal{F}^0_t),\underline{P}_t,\underline{Q}_t,\hat{u}^0_t,\hat{u}_t)dt+Q^{10}_t dW^0_t+Q^{11}_t dW_t\\
&-\mathbb{E}^{\mathcal{F}^0_t}[\partial_\mu H_0(t,\tilde{\underline{X}}_t,\mathcal{L}(\tilde{X}_t\vert\mathcal{F}^0_t),\tilde{\underline{P}}_t,\tilde{\underline{Q}}_t,\tilde{\hat{u}}^0_t,\tilde{\hat{u}}_t)(X_t)]dt
\end{aligned}\\
dY^0_t=-\partial_{x_0} H(t,\underline{\check{X}}_t,\mu_t,\underline{Y}_t,\underline{Z}_t,\hat{u}^0_t,\hat{u}_t)dt+Z^{00}_t dW^0_t+Z^{01}_t dW_t,\\
dY_t=-\partial_x H(t,\underline{\check{X}}_t,\mu_t,\underline{Y}_t,\underline{Z}_t,\hat{u}^0_t,\hat{u}_t)dt+Z^{10}_t dW^0_t+Z^{11}_t dW_t,
\end{cases}
\end{equation}
with the initial and terminal conditions given by
$$\begin{cases}X^0_0=x^0_0,\quad
X_0=x_0,\\
P^0_T=\partial_{x_0} g(X^0_T,\mathcal{L}(X_T\vert \mathcal{F}^{0}_t)),\\
P_T=\mathbb{E}^{\mathcal{F}^0_T}[\partial_\mu g(\tilde{X}^0_T,\mathcal{L}(\tilde{X}_T\vert \mathcal{F}^0_T))(X_T)],
\end{cases},\quad \begin{cases}
\check{X}^0_0=x^0_0,\quad \check{X}_0=x_0,\\
Y^0_T=\partial_{x_0} g(\check{X}_T,\mu_T,\check{X}^0_T),\\
Y_T=\partial_x g(\check{X}_T,\mu_T,\check{X}^0_T),
\end{cases}$$
where
$$\hat{u}^0_t=\hat{u}^0(t,X^0_t,\mathcal{L}(X_t \vert \mathcal{F}^0_t),\bar{P}^0_t,\bar{Q}^{00}_t), \quad \hat{u}_t=\hat{u}(t,\check{X}^0_t,\check{X}_t,\mu_t,Y_t,Z^{11}_t).$$
If this FBSDE has a solution, then $(\hat{u}^0,\hat{u})$ is a Nash equilibrium for the limiting two-player stochastic differential game.
\end{proposition}

\subsection{The consistency condition}
The last step in the scheme amounts to imposing the consistency condition which writes
$$\mu_t=\mathcal{L}(X_t\vert \mathcal{F}^0_t), \quad \forall t \in [0,T].$$
Plugging it into FBSDE (\ref{fo:FBSDEMm}) gives the following ultimate FBSDE:
\begin{equation}\label{fo:Ultimate}
\begin{cases}
dX^0_t=\partial_{p_0} H_0(t,\underline{X}_t,\mathcal{L}(X_t\vert\mathcal{F}^0_t),\underline{P}_t,\underline{Q}_t,\hat{u}^0_t,\hat{u}_t)dt+\partial_{q_{00}}H_0(t,\underline{X}_t,\mathcal{L}(X_t\vert\mathcal{F}^0_t),\underline{P}_t,\underline{Q}_t,\hat{u}^0_t,\hat{u}_t) dW^0_t,\\
dX_t=\partial_p H_0(t,\underline{X}_t,\mathcal{L}(X_t\vert\mathcal{F}^0_t),\underline{P}_t,\underline{Q}_t,\hat{u}^0_t,\hat{u}_t)dt+\partial_{q_{11}}H_0(t,\underline{X}_t,\mathcal{L}(X_t\vert\mathcal{F}^0_t),\underline{P}_t,\underline{Q}_t,\hat{u}^0_t,\hat{u}_t) dW_t,\\
dP^0_t=-\partial_{x_0} H_0(t,\underline{X}_t,\mathcal{L}(X_t\vert\mathcal{F}^0_t),\underline{P}_t,\underline{Q}_t,\hat{u}^0_t,\hat{u}_t)dt+Q^{00}_t dW^0_t+Q^{01}_t dW_t,\\
\begin{aligned}
dP_t=&-\partial_x H_0(t,\underline{X}_t,\mathcal{L}(X_t\vert\mathcal{F}^0_t),\underline{P}_t,\underline{Q}_t,\hat{u}^0_t,\hat{u}_t)dt+Q^{10}_t dW^0_t+Q^{11}_t dW_t\\
&-\mathbb{E}^{\mathcal{F}^0_t}[\partial_\mu H_0(t,\tilde{\underline{X}}_t,\mathcal{L}(\tilde{X}_t\vert\mathcal{F}^0_t),\tilde{\underline{P}}_t,\tilde{\underline{Q}}_t,\tilde{\hat{u}}^0_t,\tilde{\hat{u}}_t)(X_t)]dt
\end{aligned}\\
dY^0_t=-\partial_{x_0} H(t,\underline{X}_t,\mathcal{L}(X_t\vert\mathcal{F}^0_t),\underline{Y}_t,\underline{Z}_t,\hat{u}^0_t,\hat{u}_t)dt+Z^{00}_t dW^0_t+Z^{01}_t dW_t,\\
dY_t=-\partial_x H(t,\underline{X}_t,\mathcal{L}(X_t\vert\mathcal{F}^0_t),\underline{Y}_t,\underline{Z}_t,\hat{u}^0_t,\hat{u}_t)dt+Z^{10}_t dW^0_t+Z^{11}_t dW_t,
\end{cases}
\end{equation}
with initial and terminal conditions given by
\begin{equation}
\begin{cases}
X^0_0=x^0_0, \quad X_0=x_0,\\
P^0_T=\partial_{x_0} g(X^0_T,\mathcal{L}(X_T \vert \mathcal{F}^0_T)),\\
P_T=\mathbb{E}^{\mathcal{F}^0_T}[\partial_\mu g(\tilde{X}^0_T,\mathcal{L}(\tilde{X}_T \vert\mathcal{F}^0_T))(X_T)],\\
Y^0_T=\partial_{x_0} g(X_T, \mathcal{L}(X_T\vert\mathcal{F}^0_T),X^0_T),\\
Y_T=\partial_x g(X_T, \mathcal{L}(X_T\vert\mathcal{F}^0_T),X^0_T).
\end{cases}
\end{equation}
where this time we define
$$\hat{u}^0_t=\hat{u}^0(t,X^0_t,\mathcal{L}(X_t\vert\mathcal{F}^0_t),\bar{P}^0_t,\bar{Q}^{00}_t), \quad \hat{u}_t=\hat{u}(t,X^0_t,X_t,\mathcal{L}(X_t\vert\mathcal{F}^0_t),Y_t,Z^{11}_t).$$
\begin{remark}
Note that after implementing the consistency condition, $(X^0,X)$ and $(\check{X}^0,\check{X})$ become the same. We can also check that if we replace the current consistency condition by
$$\mu_t=\mathcal{L}(\check{X}_t\vert \mathcal{F}^0_t), \quad \forall t\in[0,T]$$
we arrive at the same FBSDE as above.
\end{remark}
\begin{remark}\label{rk:redundant}
In the limiting control problem faced by the representative minor player, the dynamic of the major player is not affected by the control $u$ and can be considered given. As a result, the adjoint process $Y^0$ is redundant and independent of the rest of the system, and could have been discarded from the system (\ref{fo:Ultimate}). It is there in (\ref{fo:Ultimate}) because we want to write the system in a symmetric and compact fashion using the Hamiltonians $H_0$ and $H$.
\end{remark}

The solvability of conditional McKean-Vlasov FBSDEs in the form of (\ref{fo:Ultimate}) is a hard problem. If the conditional distributions in (\ref{fo:Ultimate}) are replaced by plain distributions, the resulting FBSDEs are usually called ``mean field FBSDEs'' and are studied in some recent papers, see for example \cite{CarmonaDelarue_ecp}. The conditioning with respect to $\mathcal{F}^0_t$ makes (\ref{fo:Ultimate}) substantially harder to solve compared to the ones already considered in the literature, and we leave the well-posedness of FBSDEs of the form of (\ref{fo:Ultimate}) to future research.

\section{Propagation of chaos and $\epsilon$-Nash equilibrium}
\label{se:approximate}
In this section we prove a central result stating that, when we apply the optimal control law found in the limiting regime to all the players in the original $N$-player game, we will find an approximate Nash equilibrium. This justifies the whole scheme as an effective way to find approximate Nash equilibria for the finite-player games. Throughout this section we assume that (A1-4), (M), (M0), (C) and (C0) hold. In addition, we assume that

\noindent (\textbf{A5}) The diffusion coefficients $\sigma_0$ and $\sigma$ are constants.

Assumption (A5) is too strong for what we really need. We should merely assume that the two volatility $\sigma_0$ and $\sigma$ are independent of the controls $u^0$ and $u$. All the derivations given below can be adapted to this more general setting, but in order to limit the complexity of the formulas appearing in the arguments, we limit ourselves to assumption (A5).

Let's first recall the finite-player game setup under the assumption (A5): the controlled dynamics are now given by
\begin{equation}\label{fo:SDEfinite}
\begin{cases}
dX^{0,N}_t=b_0(t,X^{0,N}_t,\mu^N_t,u^{0,N}_t)dt+\sigma_0 dW^0_t, \quad X^{0,N}_0=x^0_0,\\
dX^{i,N}_t=b(t,X^{i,N}_t,\mu^N_t,X^{0,N}_t,u^{i,N}_t)dt+\sigma dW^i_t, \quad X^{i,N}_0=x_0, \quad i=1,2,...,N,
\end{cases}
\end{equation}
and the cost functionals by
$$\begin{aligned}
&J^{0,N}=\mathbb{E}\left[\int^T_0 f_0(t,X^{0,N}_t,\mu^N_t,u^{0,N}_t)dt+g_0(X^{0,N}_T,\mu^N_T)\right],\\
&J^{i,N}=\mathbb{E}\left[\int^T_0 f(t,X^{i,N}_t,\mu^N_t,X^{0,N}_t,u^{i,N}_t)dt+g(X^{i,N}_T,\mu^N_T,X^{0,N}_T)\right], 1 \leq i \leq N.
\end{aligned}$$
The sets of admissible controls for this $(N+1)$-player game are defined as follows.

\begin{definition}
In the above $(N+1)$-player game, a process $u^{0,N}$ is said to be admissible for the major player if  $u^{0,N}\in\HH^{2,d_0}(\FF^0,U_0)$ and it is said to be $\kappa$-admissible for the major player if additionally we have

\begin{equation}
\label{fo:adm}
\mathbb{E}\left[\int^T_0 \vert u^{0,N}_t \vert^p\right] \leq \kappa.
\end{equation}
with $i=0$ and $p=d+5$. On the other hand, a process $u^{i,N}$ is said to be admissible for the $i$-th minor player if $u^{1,N}\in\HH^{2,d}(\FF^{W^0,W^1,\cdots,W^N},U)$, and  $\kappa$-admissible for the $i$-th minor player if additionally it satisfies \eqref{fo:adm} with $p=2$.
The set of admissible controls and $\kappa$-admissible controls for the $i$-th player are respectively denoted by $\mathcal{A}_i$ and $\mathcal{A}^\kappa_i$, $i \geq 0$. Note that $\mathcal{A}_i$ and $\mathcal{A}^\kappa_i$ are independent of $i \geq 1$.
\end{definition}

Note that due to (A1-3), for all $(u^{0,N}, u^{1,N},..., u^{N,N}) \in \prod^N_{i=0}\mathcal{A}_i$, the controlled SDE (\ref{fo:SDEfinite}) always has a unique solution. On the other hand, we will see that the notion of $\kappa$-admissible controls plays an important role in Theorem \ref{th:CLT} to obtain a quantitative uniform speed of convergence. We then give the definition of $\epsilon$-Nash equilibrium in the context of the above finite-player game.
\begin{definition}
A set of admissible controls $(u^{0,N},u^{1,N},...,u^{N,N}) \in \prod^N_{i=0}\mathcal{A}_i$ is called an $\epsilon$-Nash equilibrium in $\mathcal{A}^\kappa_0 \times \prod^N_{i=1}\mathcal{A}^\kappa_i$ for the above $(N+1)$-player stochastic differential game if for all $u^{0} \in \mathcal{A}^\kappa_0$ we have
$$J^{0,N}(u^{0,N},u^{1,N},...,u^{N,N})-\epsilon \leq J^{0,N}(u^0,u^{1,N},...,u^{N,N}),$$
and for all $1 \leq i \leq N$ and $u \in \mathcal{A}^\kappa_i$ we have
$$J^{i,N}(u^{0,N},u^{1,N},...,u^{N,N})-\epsilon \leq J^{i,N}(u^{0,N},...,u^{i-1,N},u,u^{i+1,N},...,u^{N,N}).$$
\end{definition}

The following lemma is useful to derive explicit bounds on the rate of convergence of approximate Nash equilibrium. In order to obtain a quantitative convergence estimate, we rely on the following result of Horowitz and Karandikar  which can be found in \cite{RachevRuschendorf}.
\begin{lemma}
\label{le:RR}
Let $(X_n)$ be a sequence of exchangeable random variables taking values in $\mathbb{R}^d$ with directing (random) measure $\mu$ satisfying
$$
c:=\int \vert u \vert^{d+5}\beta(du) < \infty.
$$
where $\beta$ is the marginal of $\mu$ in the sense that $\beta(A)=\EE[\mu(A)]$.
Then there exists a constant $C$ depending only upon $c$ and $d$ such that
$$
\mathbb{E}[W^2_2(\mu^N,\mu)]\leq c N^{-2/(d+4)},$$
where as usual, $\mu^N$ is the empirical measure of $X_1, \cdots, X_N$.
\end{lemma}
Recall that the \emph{directing measure} of the sequence is the almost sure limit as $N\to\infty$ of the empirical measures $\mu^N$
Before stating and proving the central theorem of this section, we introduce two additional assumptions.

\noindent (\textbf{A7}) The FBSDE (\ref{fo:Ultimate}) admits a unique solution. Moreover, there exists a random decoupling field $\theta:[0,T] \times \Omega \times \mathbb{R}^{d_0} \times \mathbb{R}^d \hookrightarrow \theta(t,\omega,x_0,x)$ such that
$$Y_t=\theta(t,X^0_t,X_t),\quad \text{a.s..}$$
Finally $\theta$ satisfies:\\
(1) There exists a constant $c_\theta$ such that
$$\vert \theta(t,\omega,x'_0,x')-\theta(t,\omega,x_0,x)\vert \leq c_\theta(\vert x'_0-x_0\vert+\vert x'-x\vert).$$
(2) For all $(t,x_0,x) \in [0,T]\times\mathbb{R}^{d_0} \times \mathbb{R}^d$, $\theta(t,\cdot,x_0,x)$ is $\mathcal{F}^0_t$-measurable.

The concept of (deterministic) decoupling field lies at the core of many investigations of the well-posedness of standard FBSDEs, see for example \cite{Delarue02,MaYong}. Its non-Markovian counterpart corresponding to non-Markovian FBSDEs was introduced in \cite{MaWuZhangZhang}. The possibility of applying existing results concerning the well-posedness of non-Markovian FBSDEs is appealing, but due to the conditional McKean-Vlasov nature of FBSDE (\ref{fo:Ultimate})  a general sufficient condition is hard to come by, and it is highly likely that well-posedness can only be established on a case-by-case basis. A concrete sufficient condition of well-posedness and the existence of a decoupling field will be given in Section 5 for Linear Quadratic Gaussian (LQG for short) models.

The following theorem is the central result in this section. It stipulates that when the number of players is sufficiently large, the solution of the limiting problem provides  approximate Nash equilibriums. Note that an important consequence of assumption (A5) is that the minimizer $\hat{u}^0$ identified in the previous section is now independent of $q_{00}$, and by an abuse of notation, we use $\hat{u}^0(t,x_0)$ to denote $\hat{u}^0(t,x_0,\mathcal{L}(X_t\vert\mathcal{F}^0_t), \bar{P}^0_t)$. Accordingly, $\hat{u}$ is now independent of $z_{11}$, and if we assume that (A7) is in force, $Y_t$ can then be written as $\theta(t, X^0_t, X_t)$, and again by a similar abuse of notation we use $\hat{u}(t,x_0,x)$ to denote $\hat{u}(t,x_0,x,\mathcal{L}(X_t\vert\mathcal{F}^0_t),\theta(t,x_0,x))$, where $X$, $P^0$ solve the FBSDE (\ref{fo:Ultimate}). Finally we impose

\noindent (\textbf{A8}) There exists a constant $c$ such that for all $t \in [0,T]$ and $x'_0,x_0 \in \mathbb{R}^{d_0}$,
$$\vert \hat{u}^0(t,x'_0)-\hat{u}^0(t,x_0)\vert \leq c \Vert x'_0-x_0\Vert,\quad \text{a.s..}$$
Moreover,
$$\mathbb{E}\left[\int^T_0 \vert \hat{u}^0(t,0)\vert^2 dt\right]< \infty.$$

\begin{theorem}
\label{th:CLT}
There exists a sequence $(\epsilon_N)_{N \geq 1}$ and a non-decreasing function $\rho: \mathbb{R}^+ \rightarrow \mathbb{R}^+$ such that \\
(i) There exists a constant $c$ such that for all $N \geq 1$,
$$\epsilon_N \leq c N^{-1/(d+4)}.$$
(ii) The feedback profile $(\hat{u}^0(t,X^{0,N}_t),(\hat{u}(t,X^{0,N}_t,X^{i,N}_t))_{1 \leq i \leq N})$ forms an $(\rho(\kappa)\epsilon_N)$-Nash equilibrium for the $(N+1)$-player game when the admissible control sets are taken as $\mathcal{A}^\kappa_0 \times \prod^N_{i=1} \mathcal{A}^\kappa_i$.
\end{theorem}

\begin{proof}
For a fixed $N$, we start with investigating what happens if the major player deviates from the strategy $\hat{u}^0(t,\hat{X}^{0,N}_t)$ unilaterally.  When all the players apply the feedback controls identified in the statement of the theorem, the resulting controlled state processes will be denoted by $(\hat{X}^{i,N})_{i \geq 0}$ and solve
\begin{equation}\label{fo:hatfinite}
\begin{cases}
d\hat{X}^{0,N}_t=b_0(t,\hat{X}^{0,N}_t, \hat{\mu}^N_t, \hat{u}^0(t,\hat{X}^{0,N}_t))dt+\sigma_0 dW^0_t,\quad \hat{X}^{0,N}_0=x^0_0,\\
d\hat{X}^{i,N}_t=b(t,X^{i,N}_t,\hat{\mu}^N_t,\hat{X}^0_t, \hat{u}(t,\hat{X}^{0,N}_t,\hat{X}^{i,N}_t))dt+\sigma dW^i_t, \quad \hat{X}^{i,N}_0=x_0, \quad i \geq 1,
\end{cases}
\end{equation}
where the empirical measures are defined as in \eqref{fo:muN}.
Following the approach presented in Section \ref{se:conditional_chaos}, we define the limiting nonlinear processes as the solution of
\begin{equation}\label{fo:hatlimit}
\begin{cases}
d\hat{X}^0_t=b_0(t,\hat{X}^0_t,\mathcal{L}(\hat{X}^1_t|\mathcal{F}^0_t),\hat{u}^0(t,\hat{X}^0_t))dt+\sigma_0 dW^0_t,\quad \hat{X}^0_0=x^0_0,\\
d\hat{X}^i_t=b(t,\hat{X}^i_t,\mathcal{L}(\hat{X}^1_t|\mathcal{F}^0_t),\hat{X}^0_t,\hat{u}(t,\hat{X}^0_t,\hat{X}^i_t))dt+\sigma dW^i_t, \quad \hat{X}^{i}_0=x_0, \quad i \geq 1.
\end{cases}
\end{equation}
The stochastic measure flow $\mathcal{L}(\hat{X}^1_t\vert\mathcal{F}^0_t)$ will be sometimes denoted by $\hat{\mu}_t$ in the following. A direct application of Theorem \ref{tpropagation} in Section \ref{se:conditional_chaos} yields the existence of a constant $\hat{c}$ such that
\begin{equation}\label{fo:hatcon}
\max_{0 \leq i \leq N}\mathbb{E}\left[\sup_{0 \leq t \leq T} \vert \hat{X}^{i,N}_t-\hat{X}^i_t\vert^2\right] \leq \hat{c} N^{-2/(d+4)},
\end{equation}
and by applying the usual upper bound for 2-Wasserstein distance we also have
\begin{equation}\label{fo:hatcon2}
\mathbb{E}\left[\sup_{0 \leq t \leq T} W^2_2\left(\hat{\mu}^N_t, \frac{1}{N}\sum^N_{i=1}\delta_{\hat{X}^i_t}\right)\right] \leq \hat{c}N^{-2/(d+4)},
\end{equation}
where $\hat{c}$ depends upon $T$, the Lipschitz constants of $b_0$, $b$, $\hat{u}^0$ and $\hat{u}$, and
$$\hat{\eta}=\mathbb{E}\int^T_0 \vert \hat{X}^1_t\vert^{d+5} dt.$$
Now we turn our attention to the cost functionals. We define
$$\begin{aligned}
&\hat{J}^{0,N}=\mathbb{E}\left[\int^T_0 f_0(t,\hat{X}^{0,N}_t,\hat{\mu}^N_t,\hat{u}^0(t,\hat{X}^{0,N}_t))dt+g_0(\hat{X}^{0,N}_T,\hat{\mu}^N_T)\right],\\
&\hat{J}^{0}=\mathbb{E}\left[\int^T_0 f_0(t,\hat{X}^0_t,\hat{\mu}_t,\hat{u}^0(t,\hat{X}^0_t))dt+g_0(\hat{X}^0_T,\hat{\mu}_T)\right],
\end{aligned}$$
and we have, by assumptions (A3) and (A7), that
\begin{equation}\label{fo:hatcost}
\begin{aligned}
&\vert\hat{J}^{0,N}-\hat{J}^0\vert
 = \bigg| \mathbb{E}\left[\int^T_0 f_0(t,\hat{X}^{0,N}_t,\hat{\mu}^N_t,\hat{u}^0(t,\hat{X}^{0,N}_t))+g_0(\hat{X}^{0,N}_T,\hat{\mu}^N_T)\right]\\
   &-\mathbb{E}\left[\int^T_0 f_0(t,\hat{X}^0_t,\hat{\mu}_t,\hat{u}^0(t,\hat{X}^0_t))+g_0(\hat{X}^0_T, \mu_T)\right] \bigg| \\
 \leq & \mathbb{E}\int^T_0 c\left( 1+\vert \hat{X}^{0,N}_t\vert+\vert \hat{X}^0_t\vert +\vert\hat{u}^0(t,\hat{X}^{0,N}_t)\vert+\vert \hat{u}^0(t,\hat{X}^0_t)\vert+M_2(\hat{\mu}^N_t)+M_2(\hat{\mu}_t)\right)\\
 & \quad \quad \quad \left(\vert \hat{X}^{0,N}_t-\hat{X}^0_t\vert+W_2(\hat{\mu}^N_t,\hat{\mu}_t)\right)dt\\
 \leq & c\mathbb{E}\left[\int^T_0 1+\vert \hat{X}^{0,N}_t\vert^2+\vert \hat{X}^0_t\vert^2+\frac{1}{N}\sum^N_{i=1}\vert \hat{X}^{i,N}_t\vert^2+\vert \hat{X}^1_t\vert^2 dt\right]^{1/2} \mathbb{E}\left[\int^T_0 \vert \hat{X}^{0,N}_t-\hat{X}^0_t\vert^2+W^2_2(\hat{\mu}^N_t,\hat{\mu}_t) dt\right]^{1/2}
\end{aligned}
\end{equation}
and by applying (\ref{fo:hatcon}) and (\ref{fo:hatcon2}) we deduce that
\begin{equation}\label{fo:hatJ}
\hat{J}^{0,N}=\hat{J}^0+O(N^{-1/(d+4)}).
\end{equation}
Assume now that the major player uses a different admissible control $v^0 \in \mathcal{A}^\kappa_0$, and other minor players keep using the strategies $(\hat{u}(t,\hat{X}^{i,N}_t))_{i \geq 1}$. The resulting perturbed state processes will be denoted by $(\tilde{X}^{i,N}_t)_{i \geq 0}$ and is the solution of the system
\begin{equation}\label{fo:tildefinite}
\begin{cases}
d\tilde{X}^{0,N}_t=b_0(t,\tilde{X}^{0,N}_t,\tilde{\mu}^N_t, v^0_t)dt+\sigma_0 dW^0_t, \quad \tilde{X}^{0,N}_0=x^0_0,\\
d\tilde{X}^{i,N}_t=b(t,\tilde{X}^{i,N}_t,\tilde{\mu}^N_t,\tilde{X}^{0,N}_t,\hat{u}(t,\hat{X}^{i,N}_t))dt+\sigma dW^i_t, \quad \tilde{X}^{i,N}_0=x_0, \quad 1 \leq i \leq N,
\end{cases}
\end{equation}
where as usual, $\tilde{\mu}^N_t$ denotes the empirical distribution of the $\tilde{X}^{i,N}_t$. 
Note that $\hat{X}^{i,N}$ is not $\mathcal{F}^0_t$-progressively measurable in general, in order to apply Theorem \ref{tpropagation} we combine (\ref{fo:hatfinite}) and (\ref{fo:tildefinite}) and consider the limiting nonlinear processes defined as the solution of
\begin{equation}
\begin{cases}
d\hat{X}^0_t=b_0(t,\hat{X}^0_t,\mathcal{L}(\hat{X}^1_t|\mathcal{F}^0_t),\hat{u}^0(t,\hat{X}^0_t))dt+\sigma_0 dW^0_t, \quad \hat{X}^{0}_0=x^0_0,\\
d\hat{X}^i_t=b(t,\hat{X}^i_t,\mathcal{L}(\hat{X}^i_t|\mathcal{F}^0_t),\hat{X}^0_t,\hat{u}(t,\hat{X}^i_t))dt+\sigma dW^i_t, \quad \hat{X}^i_0=x_0, \quad i \geq 1,\\
d\tilde{X}^{0}_t=b_0(t,\tilde{X}^0_t,\mathcal{L}(\tilde{X}^i_t\vert \mathcal{F}^0_t),v^0_t)dt+\sigma_0 dW^0_t,\quad \tilde{X}^0_0=x^0_0,\\
d\tilde{X}^{i}_t=b(t,\tilde{X}^i_t,\mathcal{L}(\tilde{X}^i_t\vert \mathcal{F}^0_t),\tilde{X}^0_t,\hat{u}(t,\hat{X}^i_t))dt+\sigma dW^i_t, \quad \tilde{X}^i_0=x_0, \quad i\geq 1,
\end{cases}
\end{equation}
and now Theorem \ref{tpropagation} yields the existence of a constant $\tilde{c}$ such that
$$\mathbb{E}\left[\sup_{0 \leq t \leq T}\vert \tilde{X}^{i,N}_t-\tilde{X}^i_t\vert^2\right] \leq \tilde{c}N^{-2/(d+4)},$$
where $\tilde{c}$ depends upon $T$, the Lipschitz constants of $b_0$, $b$, $\hat{u}^0$, $u$, $\hat{\eta}$ and
$$\tilde{\eta}=\mathbb{E}\int^T_0 \vert \tilde{X}^1_t\vert^{d+5} dt.$$
It is important to note that $\tilde{\eta}$ depends on the control $v^0$. On the other hand the coefficients $b_0$ and $b$ are globally Lipschitz-continuous, so by usual estimates and Gronwall's inequality, for all $\kappa >0$ there exists a constant $\rho_1(\kappa)$ such that
$$\mathbb{E}\int^T_0 \vert v^0_t\vert^{d+5} dt \leq \kappa \Longrightarrow \tilde{\eta} \leq \rho^0_1(\kappa).$$
It is then clear that for all $\kappa>0$ there exists a constant $\rho_2(\kappa)$ such that
$$\mathbb{E}\int^T_0 \vert v^0_t\vert^{d+5} dt \leq \kappa \Longrightarrow \tilde{c} \leq \rho^0_2(\kappa).$$
By using the same estimates as in (\ref{fo:hatcost}), we deduce that there exists a constant $\rho(\kappa)$ such that for all $v^0 \in \mathcal{A}^\kappa_0$, we have
\begin{equation}\label{fo:tildeJ}
\vert\tilde{J}^{0,N}-\tilde{J}^0\vert \leq \rho(\kappa)\epsilon_N N^{-1/(d+4)}.
\end{equation}
Finally, since $(\hat{u}^0(t,\hat{X}^0_t),\hat{u}(t,\hat{X}_t))$ solves the limiting two-player game problem, it is clear that
\begin{equation}\label{fo:limitJ}
\hat{J}^0 \leq \tilde{J}^0,
\end{equation}
and combining (\ref{fo:hatJ}), (\ref{fo:tildeJ}) and (\ref{fo:limitJ}) we get the desired result for the major player.

We then consider the case when a minor player changes his strategy unilaterally, and without loss of generality we consider the case when the minor player with index 1 changes his strategy to $v \in \mathcal{A}_1$. This part of the proof is highly similar with that of Theorem 3 in \cite{CarmonaDelarue_sicon}, and we will refer to \cite{CarmonaDelarue_sicon} for some details of the proof in the following. The resulting perturbed controlled dynamics are given by
$$\begin{cases}
d\bar{X}^{0,N}_t=b_0(t,\bar{X}^{0,N}_t,\bar{\mu}^N_t,\hat{u}^0(t,\hat{X}^{0,N}_t))dt+\sigma_0 dW^0_t, \quad \bar{X}^{0,N}_0=x^0_0,\\
d\bar{X}^{1,N}_t=b(t,\bar{X}^{1,N}_t,\bar{\mu}^N_t,\bar{X}^{0,N}_t,v_t)dt+\sigma dW^1_t, \quad \bar{X}^{1,N}_0=x_0,\\
d\bar{X}^{i,N}_t=b(t,\bar{X}^{i,N}_t,\bar{\mu}^N_t,\bar{X}^{0,N}_t,\hat{u}(t,\hat{X}^{i,N}_t))dt+\sigma dW^i_t, \quad \bar{X}^{i,N}_0=x_0, \quad 2 \leq i \leq N.
\end{cases}$$
By the usual estimates on the difference between $\bar{X}^{i,N}$ and $\hat{X}^{i,N}$, and by applying Gronwall's inequality we can show that
\begin{equation}
\mathbb{E}\left[\sup_{0 \leq t \leq T} \vert \bar{X}^{0,N}_t-\hat{X}^{0,N}_t\vert^2\right]+\frac{1}{N}\sum^N_{i=1} \mathbb{E}\left[\sup_{0 \leq t \leq T} \vert \bar{X}^{i,N}_t-\hat{X}^{i,N}_t\vert^2\right]
\leq  \frac{c}{N}\int^T_0 \vert v_t - \hat{u}(t,\hat{X}^{1,N}_t)\vert^2 dt.
\end{equation}
Combining the above bound, the growth properties of $\hat{u}$ and (\ref{fo:hatcon}), we see that for all $\kappa >0$, there exists a non-decreasing function $\rho_1: \mathbb{R}^+ \rightarrow \mathbb{R}^+$ such that
$$\int^T_0 \vert v_t\vert^2 \leq \kappa \quad \Rightarrow \quad \mathbb{E}\left[ \sup_{0 \leq t \leq T} \vert \bar{X}^{0,N}_t -\hat{X}^0_t\vert^2\right] +\mathbb{E}\left[ \sup_{0 \leq t \leq T}W^2_2(\bar{\mu}_t,\mu_t)\right]\leq \rho_1(\kappa)N^{-2/(d+4)}.$$
We hence conclude that there exists a non-decreasing function $\rho_2: \mathbb{R}^+\rightarrow \mathbb{R}^+$ such that when $\int^T_0 \vert v_t\vert^2 \leq \kappa$, we have
$$\mathbb{E}\left[\sup_{0 \leq t \leq T} \vert \bar{X}^{1,N}_t-\bar{X}^1_t\vert^2\right] \leq \rho_2(\kappa) N^{-2/(d+4)},$$
where $\bar{X}^1$ is the solution of the SDE
\begin{equation}
d\bar{X}^1_t=b(t,\bar{X}^1_t,\mu_t,X^0_t,v_t)dt+\sigma dW^1_t, \quad \bar{X}^1_0=x_0,
\end{equation}
where $\mu$ and $X^0$ are in the solution of the FBSDE (\ref{fo:Ultimate}). We then conclude in the same way as for the major player.
\end{proof}

\section{MFG with Major-Minor Agents: the LQG Case}
\label{se:lqg}
The linear-quadratic-gaussian (LQG) stochastic control problems are among the best-understood models in stochastic control theory. It is thus natural to expect explicit results for the major-minor mean field games in a similar setting.  This type of model was first treated in \cite{Huang} in infinite horizon. The finite-horizon case was treated in \cite{NguyenHuang1}. However, the state of the major player does not enter the dynamics of the states of the minor players in \cite{NguyenHuang1}. The general finite-horizon case is solved in \cite{NguyenHuang2} by the use of the so-called \emph{nonanticipative variational calculus}. It is important to point out that the notion of Nash equilibrium used in \cite{NguyenHuang2} corresponds to the \emph{Markovian feedback Nash equilibrium'} while here, we work with open-loop Nash equilibriums. In what follows, we carry out the general systematic scheme introduced in the previous discussions and derive approximate Nash equilibria for the LQG major-minor mean field games. 

The dynamics of the states of the players are given by the following linear SDEs:
$$
\begin{cases}
dX^{0,N}_t=(A_0 X^{0,N}_t+B_0 u^{0,N}_t+F_0 \bar{X}^N_t)dt+D_0 dW^0_t,\\
dX^{i,N}_t=(AX^{i,N}_t+Bu^{i,N}_t+F\bar{X}^N_t+GX^0_t)dt+DdW^i_t.
\end{cases}
$$
For the sake of presentation we introduce the linear transformations $\Phi$ and $\Psi$ defined by:
$$
\Phi(X)=H_0 X+\eta_0, 
\quad\text{ and }\quad
\Psi(X,Y)=H X+\hat{H}Y+\eta.
$$
The cost functionals for the major and minor players are given by
$$J^0(u)=\mathbb{E}\left[\int^T_0\left\{(X^0_t-\Phi(\bar{X}^N_t))^\dagger Q_0 (X^0_t-\Phi(\bar{X}^N_t))+u^{0\dagger}_t R_0 u^0_t\right\}dt\right],$$
$$J^{i,N}(u)=\mathbb{E}\left[\int^T_0 \left\{(X^{i,N}_t-\Psi(X^0_t,\bar{X}^N_t))^\dagger Q (X^{i,N}_t-\Psi(X^0_t,\bar{X}^N_t))+u^{i,N\dagger}_t R u^{i,N}_t\right\}dt\right],$$
in which $Q$, $Q_0$, $R$ and $R_0$ are symmetric matrices and $R$ and $R_0$ are assumed to be positive definite. We use the notation $a^\dagger$ for the transpose of $a$.

We check readily that all previously mentioned assumptions hold in the above LQG setting. We then arrive directly at the non-Markovian conditional McKean-Vlasov FBSDE (\ref{fo:Ultimate}) which writes (note Remark \ref{rk:redundant})
\begin{equation}\label{fo:UltimateLQG}
\begin{cases}
dX^0_t=(A_0 X^0_t-\frac{1}{2}B_0 R^{-1}_0 B^\dagger_0 \mathbb{E}[P^0_t\vert \mathcal{F}^0_t]+F_0 \mathbb{E}[X_t \vert \mathcal{F}^0_t])dt+D_0 dW^0_t,\\
dX_t=(AX_t-\frac{1}{2}BR^{-1}B^\dagger Y_t+F\mathbb{E}[X_t\vert \mathcal{F}^0_t]+GX^0_t)dt+DdW_t,\\
dP^0_t=(-A^\dagger_0 P^0_t-G^\dagger P_t-2Q_0(X^0_t-\Phi(\mathbb{E}[X_t \vert \mathcal{F}^0_t])))dt+Q^{00}_t dW^0_t+Q^{01}_t dW_t,\\
\begin{aligned}
dP_t=&-A^\dagger P_t+Q^{10}_t dW^0_t+Q^{11}_t dW_t\\
&-F^\dagger_0\mathbb{E}[P^0_t\vert \mathcal{F}^0_t]dt-F^\dagger\mathbb{E}[ P_t\vert \mathcal{F}^0_t]dt-2 H_0^\dagger Q_0(X^0_t-\Phi(\mathbb{E}[X_t\vert \mathcal{F}^0_t]))dt,
\end{aligned}\\
dY_t=(-A^\dagger Y_t-2Q(X_t-\Psi(X^0_t,\mathbb{E}[X_t \vert \mathcal{F}^0_t])))dt+Z^0_t dW^0_t+Z_t dW_t,
\end{cases}
\end{equation}
with the initial and terminal conditions given by
$$X^0_0=x^0_0, \text{ } X_0=x_0, \text{ } P^0_T=P_T=Y_T=0.$$
As already explained at the end of Section 3, the solvability of general conditional McKean-Vlasov FBSDEs is a difficult problem. However, due to the special linear structure of (\ref{fo:UltimateLQG}) we can go a step further and look for more explicit sufficient conditions of well-posedness. As before, we use a bar to denote the conditional expectation with respect to $\cF^0_t$, so we arrive at the following more compact form:
\begin{equation}\label{fo:UltimateLQG2}
\begin{cases}
dX^0_t=(A_0 X^0_t-\frac{1}{2}B_0 R^{-1}_0 B^\dagger_0 \bar{P}^0_t+F_0 \bar{X}_t)dt+D_0 dW^0_t,\\
dX_t=(AX_t-\frac{1}{2}BR^{-1}B^\dagger Y_t+F\bar{X}_t+GX^0_t)dt+DdW_t,\\
dP^0_t=(-A^\dagger_0 P^0_t-G^\dagger P_t-2Q_0 X^0_t+2Q_0H_0\bar{X}_t+2Q_0\eta_0)dt+Q^{00}_t dW^0_t+Q^{01}_t dW_t,\\
\begin{aligned}
dP_t=&-A^\dagger P_t+Q^{10}_t dW^0_t+Q^{11}_t dW_t\\
&-F^\dagger_0 \bar{P}^0_t dt-F^\dagger\bar{P}_t dt-(2H^\dagger_0 Q_0 X^0_t-2 H^\dagger_0 Q_0 H_0 \bar{X}_t -2 H^\dagger_0 Q_0 \eta_0)dt
\end{aligned}\\
dY_t=(-A^\dagger Y_t-2QX_t+2QHX^0_t+2Q\hat{H}\bar{X}_t+2Q\eta)dt+Z^0_t dW^0_t+Z_t dW_t,
\end{cases}
\end{equation}

We then condition all the equations by the filtration $\mathcal{F}^0_t$. The following lemma will be useful when we deal with the Ito stochastic integral terms.
\begin{lemma}
Let $\mathcal{F}_t$ be a filtration and $B$ a $\mathcal{F}_t$-Brownian motion. Let $H$ be a $\mathcal{F}_t$-progressively measurable process, then
$$\mathbb{E}\left[\int^T_0 H_t dB_t\vert \mathcal{F}_T\right]=\int^T_0 \mathbb{E}\left[ H_t \vert \mathcal{F}_t\right]dB_t.$$
\end{lemma}
We then use this lemma to derive the SDEs satisfied by the conditional versions of the above processes. We add a bar on the various processes to denote the conditional version, and since $X^0_t$ is already $\mathcal{F}^0_t$-adapted, its notation will stay unchanged.
\begin{equation}\label{fo:UltimateLQGbar}
\begin{cases}
dX^0_t=(A_0 X^0_t-\frac{1}{2}B_0 R^{-1}_0 B^\dagger_0 \bar{P}^0_t+F_0 \bar{X}_t)dt+D_0 dW^0_t,\\
d\bar{X}_t=(A\bar{X}_t-\frac{1}{2}BR^{-1}B^\dagger \bar{Y}_t+F \bar{X}_t+GX^0_t)dt,\\
d\bar{P}^0_t=(-A^\dagger_0 \bar{P}^0_t-G^\dagger \bar{P}_t-2Q_0 X^0_t+2Q_0 H_0 \bar{X}_t+2Q_0 \eta_0)dt+\bar{Q}^{00}_t dW^0_t,\\
\begin{aligned}
d\bar{P}_t=&-A^\dagger \bar{P}_t+\bar{Q}^{10}_t dW^0_t\\
&-F^\dagger_0 \bar{P}^0_t-F^\dagger \bar{P}_t-(2H^\dagger_0 Q_0 X^0_t-2H^\dagger_0 Q_0 H_0 \bar{X}_t-2H^\dagger_0 Q_0 \eta_0)dt
\end{aligned}\\
d\bar{Y}_t=(-A^\dagger \bar{Y}_t-2Q \bar{X}_t+2QHX^0_t+2Q\hat{H}\bar{X}_t+2Q\eta)dt+\bar{Z}^0_t dW^0_t.\\
\end{cases}
\end{equation}
If we use $\mathbf{X}$ to denote $(X^0,\bar{X})$ and $\mathbf{Y}$ for $(\bar{P}^0,\bar{P},\bar{Y})$, we can write the above FBSDE in the following standard form
\begin{equation}\label{fo:compact}
\begin{cases}
d\mathbf{X}_t=(\mathbb{A}\mathbf{X}_t+\mathbb{B}\mathbf{Y}_t+\mathbb{C})dt+\mathbb{D}dW^0_t,\\
d\mathbf{Y}_t=-(\hat{\mathbb{A}}\mathbf{X}_t+\hat{\mathbb{B}}\mathbf{Y}_t+\hat{\mathbb{C}})dt+ \mathbf{Z}_tdW^0_t,
\end{cases}
\end{equation}
with initial and terminal conditions given by
$$\mathbf{X}_0=\begin{pmatrix}
x^0_0\\
x_0
\end{pmatrix}, \text{ } \mathbf{Y}_T=\begin{pmatrix}
0\\
0\\
0
\end{pmatrix},$$
in which
$$\mathbb{A}=\begin{pmatrix}
A_0 & F_0\\
G & A+F
\end{pmatrix},
\mathbb{B}=\begin{pmatrix}
-\frac{1}{2}B_0 R^{-1}_0 B^\dagger_0 & 0 & 0\\
0 & 0 & -\frac{1}{2}BR^{-1}B^\dagger
\end{pmatrix},
\mathbb{D}=\begin{pmatrix}
D_0\\
0
\end{pmatrix},$$
$$\hat{\mathbb{A}}=\begin{pmatrix}
2Q_0 & -2Q_0 H_0 \\
2H^\dagger_0 Q_0 & -2H^\dagger_0 Q_0 H_0 \\
-2QH & 2Q-2Q\hat{H}
\end{pmatrix},
\hat{\mathbb{B}}=\begin{pmatrix}
A^\dagger_0 & G^\dagger & 0\\
F^\dagger_0 & A^\dagger+F^\dagger & 0 \\
0 &0 & A^\dagger
\end{pmatrix}.$$

In order to find explicit sufficient conditions of the well-posedness of the linear FBSDE (\ref{fo:compact}) we follow the usual four step scheme and look for solutions in the form $\mathbf{Y}_t=S_t \mathbf{X}_t+s_t$, where $S$ and $s$ are two deterministic functions defined on $[0,T]$. Consider the following matrix Riccati equation with terminal condition:
\begin{equation}\label{fo:Riccati}
\dot{S}_t+S_t \mathbb{A}+\hat{\mathbb{B}}S_t+S_t\mathbb{B}S_t+\hat{\mathbb{A}}=0, \text{ } S_T=0,
\end{equation}
and the linear ODE
\begin{equation}\label{fo:ODE}
\dot{s}_t=-(\hat{\mathbb{B}}+S_t \mathbb{B})s_t-(\hat{\mathbb{C}}+S_t \mathbb{C}), \text{ } s_T=0.
\end{equation}
We observe that, when $S$ is well-defined, the backward ODE (\ref{fo:ODE}) is always uniquely solvable. We have the following proposition.

\begin{proposition}
If the matrix Riccati equation (\ref{fo:Riccati}) and the backward ODE (\ref{fo:ODE}) admit solutions denoted by
$$S_t=\begin{pmatrix}
S^{1,1}_t & S^{1,2}_t\\
S^{2,1}_t & S^{2,2}_t\\
S^{3,1}_t & S^{3,2}_t
\end{pmatrix}, \text{ }
s_t=\begin{pmatrix}
s^1_t \\
s^2_t \\
s^3_t
\end{pmatrix},$$
 then the FBSDE (\ref{fo:UltimateLQGbar}) is uniquely solvable. The first two components in the solution, namely $(X^0,\bar{X}^0)$, is given by the solution of the linear SDE
$$\begin{cases}
d\bar{X}^0_t=(A_0 \bar{X}^0_t-\frac{1}{2}B_0 R^{-1}_0 B^\dagger_0 (S^{1,1}_t X^0_t+S^{1,2}_t \bar{X}_t+s^2_t)+F_0 \bar{X}_t)dt+D_0 dW^0_t,\\
d\bar{X}_t=(A\bar{X}_t-\frac{1}{2}BR^{-1}B^\dagger (S^{3,1}_t X^0_t+S^{3,2}_t \bar{X}_t+s^3_t)+F\bar{X}_t+G\bar{X}^0_t)dt,\\
\end{cases}$$
with  initial conditions given by
$$X^0_0=x^0_0, \text{ }, \bar{X}_0=x_0.$$
The processes $(\bar{P}^0,\bar{P}, \bar{Y})$ are given by
$$\text{ }\bar{P}^0_t=S^{1,1}_t X^0_t+S^{1,2}_t \bar{X}_t + s^1_t,\text{ }\bar{P}_t=S^{2,1}_t X^0_t+S^{2,2}_t \bar{X}_t + s^2_t,\bar{Y}_t=S^{3,1}_t X^0_t+S^{3,2}_t \bar{X}_t + s^3_t.$$
\end{proposition}

\begin{proof}
The proof is a pure verification procedure.
\end{proof}
We now turn to the original conditional FBSDE (\ref{fo:UltimateLQG2}). Now that $X^0$, $\bar{X}_t$, $\bar{P}^0$ and $\bar{P}$ are found, we plug them into the FBSDE and it becomes a standard linear FBSDE with random coefficients. By using the fact that $X^0$, $\bar{X}_t$, $\bar{P}^0$ and $\bar{P}$ are actually solutions of linear SDEs with deterministic coefficients, we have the following proposition.

\begin{proposition}
The FBSDE (\ref{fo:UltimateLQG2}) has a unique solution. Moreover, there exist a deterministic function $K$ and a $\mathcal{F}^0_t$-progressively measurable process $k$ such that
\begin{equation}\label{fo:YX}
Y_t=K_t X_t+k_t.
\end{equation}
\end{proposition}
\begin{proof}
We plug $X^0$, $\bar{X}$, $\bar{Y}$, $\bar{P}^0$ and $\bar{P}$ into (\ref{fo:UltimateLQG2}), and we readily observe that the second and the last equations form a standard FBSDE with random coefficients. The structure of this FBSDE is standard in the sense that it can be derived from an stochastic optimal control problem, which yields (\ref{fo:YX}). We now plug all the known processes into the third and the fourth equations in (\ref{fo:UltimateLQG2}), which yields a standard BSDE whose well-posedness is well known. The processes $P^0$ and $P$ thus follow.
\end{proof}

It becomes apparent that the solvability of the Riccati equation (\ref{fo:Riccati}) plays an instrumental role in the study of the unique solvability of (\ref{fo:UltimateLQG}). In order to address this problem we first define the $(2d_0+3d) \times (2d_0+3d)$-matrix $\mathcal{B}$ as
$$\mathcal{B}=\begin{pmatrix}
\mathbb{A} & \mathbb{B}\\
\hat{\mathbb{A}} & \hat{\mathbb{B}}
\end{pmatrix}.$$
We then define $\Psi(t,s)$ as
$$\Psi(s,s)=\exp(\mathcal{B}(t-s)),$$
in other words $\Psi(t,s)$ is the propagator of the matrix ODE $\dot{X}_t=\mathcal{B}X_t$ and satisfies
$$
\frac{d}{dt}\Psi(t,s)=\mathcal{B} \Psi(t,s),
$$
with initial condition $\Psi(s,s)=I_{2d_0+3d}$. 
We further consider the block structure of $\Psi(T,t)$ and write
$$\Psi(T,t)=\begin{pmatrix}
\Gamma^{1,1}_t & \Gamma^{1,2}_t\\
\Gamma^{2,1}_t & \Gamma^{2,2}_t
\end{pmatrix}.$$
We have the following sufficient condition for the unique solvability of (\ref{fo:Riccati}).
\begin{proposition}\label{pRiccati}
If for each $t \in [0,T]$, the $(d_0+2d)\times (d_0+2d)$-matrix $\Gamma^{2,2}_t$ is invertible and the inverse is a continuous function of $t$, then
$$S_t=-\left(\Gamma^{2,2}_t\right)^{-1} \Gamma^{2,1}_t$$
solves the Riccati equation (\ref{fo:Riccati}).
\end{proposition}
The assumption in Proposition \ref{pRiccati} will be denoted by assumption (A'). The above 3 propositions tell us that if assumption (A') holds, then we can apply Theorem \ref{th:CLT}. Consequently we have
\begin{theorem}\label{tCentralLQG}
Assume that assumption (A') is in force. There exist a sequence $(\epsilon_N)_{N \geq 1}$ and a non-decreasing function $\rho:\mathbb{R}^+ \rightarrow \mathbb{R}^+$such that \\
(i) There exists a constant $c$ such that for all $N \geq 1$,
$$\epsilon_N \leq c N^{-1/(d+4)}.$$
(ii) The partially feedback profile
$$\left(-\frac{1}{2}R^{-1}_0 B^\dagger_0 (S^{1,1}_t X^{0,N}_t+S^{1,2}_t \bar{X}_t+s^1_t), (-\frac{1}{2}R^{-1}B^\dagger(K_t X^{i,N}_t+k_t))_{1 \leq i \leq N}\right)$$
 forms an $\rho(\kappa)\epsilon_N$-Nash equilibrium for the $(N+1)$-player LQG game when the sets of admissible controls are taken as $\mathcal{A}^\kappa_0 \times \prod^N_{i=1}\mathcal{A}^\kappa_i$.
\end{theorem}

\section{A Concrete Example}
\label{se:example}
The scheme proposed in this paper differs from the one proposed in \cite{NguyenHuang2,NourianCaines} as the control problem faced by the major player is here of the conditional McKean-Vlasov type, and the measure flow is endogenous to the controller. This makes the limiting problem a \emph{bona fide} two-player game instead of a succession of two consecutive standard optimal control problems. Essentially, this adds another fixed point problem, coming from the Nash equilibrium for the two-player game, on top of the fixed point problem of step 3 of the standard mean field game paradigm. The reader may wonder whether after solving the two fixed point problems of the current scheme, we could end up with the same solution as in the scheme proposed in \cite{NguyenHuang2,NourianCaines}. In order to answer this question, we provide a concrete example, in which we show that the two solutions are different, and the Nash equilibria for finite-player games indeed converge to the solution of the scheme proposed in this paper.

We consider the $(N+1)$-player game whose state dynamics are given by
$$
\begin{cases}
dX^{0,N}_t=(\frac{a}{N}\sum^N_{i=1}X^{i,N}_t+b u^{0,N}_t)dt+ D_0 dW^0_t,\quad X^{0,N}_0=x^0_0,\\
dX^{i,N}_t=c X^{0,N}_t dt+ D dW^i_t, \quad X^{i,N}_0=x_0, \quad i=1,2,...,N,
\end{cases}
$$
the objective function of the major player is given by
$$
J^{0,N}=\mathbb{E}\bigg[\int^T_0  \left(q | X^{0,N}_t|^2+|u^{0,N}_t|^2\right) dt\bigg],
$$
and the objective functions of the minor players are given by
$$
J^{i,N}=\mathbb{E}\bigg[\int^T_0 | u^{i,N}_t|^2 dt\bigg].
$$
All the processes considered in this section one-dimensional. We search for an open loop Nash equilibrium. As we can readily observe, in this finite-player stochastic differential game, the minor players' best responses are always $0$, regardless of other players' control processes. Therefore, the only remaining issue is to determine the major player's best response to the minor players using a zero control. This amounts to solving a stochastic control problem. This minimalist structure of the problem will facilitate the task of differentiating the current scheme from those of \cite{NguyenHuang2,NourianCaines}.

\subsection{Finite-player Game Nash Equilibrium}
We use the stochastic maximum principle. The admissible controls for the major player are the square-integrable $\mathcal{F}^0_t$-progressively measurable processes. His Hamiltonian is given by
$$
H=y_0 (\frac{a}{N}\sum^N_{i=1} x_i +b u_0)+cx_0\sum^N_{i=1} y_i+q x^2_0+u^2_0.
$$
The minimization of the Hamiltonian is straightforward. We get $\hat u_0=-by_0/2$. Applying the \emph{game version} of the Pontryagin stochastic 
maximum principle leads to the FBSDE:
$$
\begin{cases}
dX^{0,N}_t=(\frac{a}{N}\sum^N_{i=1} X^{i,N}_t-\frac{1}{2}b^2 Y^{0,N}_t)dt+D_0 dW^0_t,\\
dX^{i,N}_t=cX^{0,N}_t dt+ D dW^i_t, \quad 1 \leq i \leq N,\\
dY^{0,N}_t=-(c\sum^N_{i=1}Y^{i,N}_t +2q X^{0,N}_t)dt+ \sum^{N}_{j=0} Z^{0,j,N}_t dW^j_t,\\
dY^{i,N}_t=-\frac{a}{N}Y^{0,N}_t dt + \sum^{N}_{j=0} Z^{i,j,N}_t dW^j_t, \quad 1 \leq i \leq N.
\end{cases}
$$
The initial conditions for the state processes are the same as always, and will be omitted systematically in the following. 
The terminal conditions read $Y^{i,N}_T=0$ for $0\le i\le N$.
Keeping in mind the fact that the optimal control identified by the necessary condition of the Pontryagin stochastic maximum principle is 
$\hat u^{0,N}_t=-bY^{0,N}_t/2$ it is clear  that, what matters in the above equations, is the aggregate behavior of the processes 
$(X^{i,N})$ and $(Y^{i,N})$. Accordingly we introduce
$$
X^N_t=\frac{1}{N}\sum^N_{i=1}X^{i,N}_t, \quad Y^N_t=\sum^N_{i=1} Y^{i,N}_t,
$$
and the above FBSDE leads to the system:
$$
\begin{cases}
dX^{0,N}_t=(a X^N_t-\frac{1}{2}b^2 Y^{0,N}_t)dt+ D_0 dW^0_t,\\
d X^N_t=cX^{0,N}_t dt+ \frac{D}{N} d(\sum^N_{i=1}W^i_t),\\
dY^{0,N}_t=-(c  Y^N_t +2q X^{0,N}_t)dt+ \sum^{N}_{j=1} Z^{0,j,N}_t dW^j_t,\\
d Y^N_t=-a Y^{0,N}_t dt+\sum^N_{i=1} \sum^N_{j=0} Z^{i,j,N}_t dW^j_t,
\end{cases}
$$
and by conditioning with respect to $\mathcal{F}^0_t$ for the last two equations we have
$$
\begin{cases}
dX^{0,N}_t=(aX^N_t-\frac{1}{2}b^2 \bar{Y}^{0,N}_t)dt+ D_0 dW^0_t,\\
dX^N_t=cX^{0,N}_t dt+ \frac{D}{N} d(\sum^N_{i=1}W^i_t),\\
d\bar{Y}^{0,N}_t=-(c \bar{Y}^N_t +2q X^{0,N}_t)dt+ Z^{0,0}_t dW^0_t,\\
d\bar{Y}^N_t=-a \bar{Y}^{0,N}_t dt+\sum_{i} Z^{i,0}_t dW^0_t,
\end{cases}
$$
where we used an over line on top of a random variable to denote its conditional expectation with respect to $\mathcal{F}^0_t$.
by following the usual scheme of solving FBSDEs we see that the solvability of the above FBSDE depends on the solvability of
\begin{equation}\label{fo:Riccati2}
\dot{S}_t+S_t A+\hat{B}P_t+P_t B P_t +\hat{A}=0,\quad S_T=0,
\end{equation}
where we define
$$A=\begin{pmatrix}
0 & a \\
c & 0
\end{pmatrix}, B=\begin{pmatrix}
-\frac{b^2}{2} & 0\\
0 & 0
\end{pmatrix}, \hat{A}=\begin{pmatrix}
2q & 0\\
0 & 0
\end{pmatrix}, \hat{B}=\begin{pmatrix}
0 & c \\
a & 0
\end{pmatrix},$$
and $S_t$ is a $2 \times 2$ matrix which can be decomposed as
$$S=\begin{pmatrix}
S^{0,0}_t & S^{0,1}_t\\
S^{1,0}_t & S^{{1,1}}_t
\end{pmatrix}.$$
If the Riccati equation (\ref{fo:Riccati2}) is uniquely solvable, we solve the following forward SDE
$$\begin{cases}
dX^{0,N}_t=(aX^N_t-\frac{1}{2}b^2 (S^{0,0}_t X^{0,N}_t+S^{0,1}_t X^N_t))dt+ D_0 dW^0_t,\\
dX^N_t=cX^{0,N}_t dt+ \frac{D}{N} d(\sum^N_{i=1}W^i_t).
\end{cases}$$
and we obtain the optimally controlled dynamic for the major player. The optimal control is given by
$$u^0_t=-\frac{b}{2}\bar{Y}^{0,N}_t.$$

\subsection{The Current Scheme}
The scheme introduced in this paper proposes to solve the McKean-Vlasov control problem consisting of the controlled dynamics
$$\begin{cases}
dX^0_t=(a \mathbb{E}[X_t\vert \mathcal{F}^0_t]+b u^0_t)dt + D_0 dW^0_t,\\
dX_t= c X^0_t dt+ D dW_t,
\end{cases}$$
the objective function remains to be
$$J^0=\mathbb{E}\int^T_0 [q (X^0_t)^2+(u^0_t)^2 ]dt.$$
Applying directly the result in the LQG part of the paper we get the FBSDE
$$\begin{cases}
dX^0_t=(a\bar{X}_t-\frac{1}{2}b^2 \bar{P}^0_t)dt+D_0 dW^0_t,\\
dX_t=c X^0_t dt+ D dW_t,\\
dP^0_t=-(2q X^0_t+cP_t)dt+Q^{00}_t dW^0_t+Q^{01}_t dW_t,\\
dP_t=-a\bar{P}^0_t dt+Q^{10}_t dW^0_t+Q^{11}_t dW_t,
\end{cases}$$
and after conditioning we get
\begin{equation}\label{fo:FBSDEnew}
\begin{cases}
dX^0_t=(a\bar{X}_t-\frac{1}{2}b^2 \bar{P}^0_t)dt+D_0 dW^0_t,\\
d\bar{X}_t=c X^0_t dt,\\
d\bar{P}^0_t=-(2q X^0_t+c\bar{P}_t)dt+\bar{Q}^{00}_t dW^0_t,\\
d\bar{P}_t=-a\bar{P}^0_t dt+\bar{Q}^{10}_t dW^0_t.
\end{cases}
\end{equation}
We still use the four-step scheme to solve this FBSDE, and we see that the associated Riccati equation is again (\ref{fo:Riccati2}). We then solve the forward SDE
$$\begin{cases}
dX^0_t=(a\bar{X}_t-\frac{1}{2}b^2 (S^{0,0}_t X^0_t+S^{0,1}_t \bar{X}_t))dt+D_0 dW^0_t,\\
d\bar{X}_t=c X^0_t dt,
\end{cases}$$
and we obtain the solution. The optimal control $u^0$ is given by $-\frac{b}{2}\bar{P}^0_t$. We have the following proposition.

\begin{proposition}\label{pnew}
For all $t \in [0,T]$ we have
$$\vert X^{0,N}_t-X^0_t\vert+ \vert X^N_t-\bar{X}_t\vert \leq e^{Kt}\frac{D}{N}\sum^N_{i=1}W^i_t.$$
As a result, we have that for all $t\in [0,T]$,
$$X^{0,N}_t \rightarrow X^0_t, X^{N}_t \rightarrow \bar{X}_t, Y^{0,N}_t \rightarrow \bar{P}^0_t, Y^N_t \rightarrow \bar{P}_t, \text{a.s.,}$$
and finally we have the convergence of the optimal controls for the finite-player games towards the limiting optimal control, namely
$$u^{0,N}_t \rightarrow u^0_t \quad \text{a.s.,}\quad \forall t \in [0,T].$$
\end{proposition}

\proof{
For a fixed $t > 0$, by calculating the difference between the SDEs satisfied by processes $X^{0,N}$, $X^N$, $X^0$ and $\bar{X}$, we see that there exists a constant $K$ such that
$$\vert X^{0,N}_t-X^0_t\vert+\vert X^N_t-\bar{X}_t\vert \leq K\int^t_0 \vert X^{0,N}_s-X^0_s\vert+\vert X^N_s-\bar{X}_s\vert ds + \frac{D}{N} \sum^N_{i = 1} \max_{0 \leq s \leq t} \vert W^i_s\vert,$$
and since the function $t \rightarrow \frac{D}{N} \sum^n_{i = 1} \max_{0 \leq s \leq t} \vert W^i_s\vert$ is increasing in $t$, we have the desired inequality. The convergence of the processes follows by letting $N$ go to infinity. 
}

\subsection{The Scheme in \cite{NguyenHuang2,NourianCaines}}
We now turn to the scheme proposed in \cite{NguyenHuang2,NourianCaines}. We start by fixing a $\mathcal{F}^0_t$-progressively measurable process $m$,
and solve the control problem consisting of the dynamics
$$
dX^0_t=(am_t+b u^0_t)dt+ D_0 dW^0_t,\quad X^0_0=x^0_0,
$$
and the objective function
$$
J^0=\mathbb{E}\int^T_0 [q (X^0_t)^2+(u^0_t)^2] dt.
$$
By applying the usual Pontryagin maximum principle we quickly arrive at the following FBSDE characterizing the optimally controlled system:
$$
\begin{cases}
dX^0_t= (a m_t-\frac{1}{2}b^2 Y^0_t)dt+D_0 dW^0_t,\\
dY^0_t=-2q X^0_t dt+ Z^0_t dW^0_t,\\
X^0_0=x^0_0,\quad
Y^0_T=0.
\end{cases}
$$
We then impose the consistency condition $m_t=\mathbb{E}\left[ X_t \vert \mathcal{F}^0_t\right] := \bar{X}_t$ which leads to
the FBSDE:
\begin{equation}
\label{fo:FBSDEold}
\begin{cases}
dX^0_t= (a \bar{X}_t-\frac{1}{2}b^2 Y^0_t)dt+D_0 dW^0_t,\\
d\bar{X}_t=c X^0_t dt,\\
dY^0_t=-2q X^0_t dt+ Z^0_t dW^0_t,\\
\end{cases}
\end{equation}
The comparison of (\ref{fo:FBSDEold}) and (\ref{fo:FBSDEnew}) will be based on the following proposition.

\begin{proposition}\label{pold}
There exists $t \in [0,T]$ and an event $E \subset \Omega$ such that $\mathbb{P}(E)>0$ and on $E$,
$$\bar{P}^0_t \neq Y^0_t.$$
\end{proposition}

\proof{
We prove this proposition by contradiction. Assume that  for all $t$, almost surely $\bar{P}^0_t = Y^0_t$. Plugging them into the first two equations of (\ref{fo:FBSDEold}) and (\ref{fo:FBSDEnew}), by uniqueness of solutions of SDEs, we know that the $X^0$ and $\bar{X}$ in these two systems are equal. Computing the difference between the third equations of (\ref{fo:FBSDEold}) and (\ref{fo:FBSDEnew}), we conclude that $\bar{P}$ is 0 by uniqueness of solutions of BSDE. Using the fourth equation in (\ref{fo:FBSDEnew}) we see that $\bar{P}$ is 0, and finally again by uniqueness of solutions of BSDE we see that $X^0$ is 0 because it is the driver in the third equation in (\ref{fo:FBSDEnew}). This is a contradiction.}

Note that the optimal control provided by the scheme in \cite{NguyenHuang2,NourianCaines} is given by $-\frac{b}{2}Y^0$. In light of Proposition \ref{pnew} and \ref{pold}, we conclude that the two schemes lead to different optimal controls, and the Nash equilibria for the finite-player games converge towards the one produced by the current scheme, instead of the one produced by the scheme proposed in \cite{NguyenHuang2,NourianCaines}.

\section{Conditional Propagation of Chaos}
\label{se:conditional_chaos}

In this section we consider a system of $(N+1)$ interacting particles with stochastic dynamics:
\begin{equation}\label{fo:finitepoc}
\begin{cases}
dX^{0,N}_t=b_0(t,X^{0,N}_t,\mu^N_t)dt+\sigma_0(t,X^{0,N}_t,\mu^N_t) dW^0_t,\\
dX^{i,N}_t=b(t,X^{i,N}_t,\mu^N_t,X^{0,N}_t)dt+\sigma(t,X^{i,N}_t,\mu^N_t,X^{0,N}_t) dW^i_t, \ \ i=1,2,...,N,\\
X^{0,N}_0=x^0_0,\quad
X^{i,N}_0=x_0, \ \ i=1,2,...,N,
\end{cases}
\end{equation}
on a probability space $(\Omega, \mathcal{F},\PP)$, where the empirical measure $\mu^N$ was defined in \eqref{fo:muN}.
Here $(W^i)_{i \geq 0}$ is a sequence of independent Wiener processes,  $W^0$ being $n_0$-dimensional and $W^i$  $n$-dimensional for $i \geq 1$. The major-particle process $X^{0,N}$ is $d_0$-dimensional, and the minor-particle processes $X^{i,N}$ are $d$-dimensional for $i \geq 1$. The coefficient functions 
$$
\begin{aligned}
&(b_0,\sigma_0) : [0,T]\times \Omega \times \mathbb{R}^{d_0}\times \mathcal{P}_2(\mathbb{R}^d) \to \mathbb{R}^{d_0}\times \mathbb{R}^{d_0\times m_0},,\\
&(b,\sigma) : [0,T] \times \Omega \times \mathbb{R}^{d} \times \mathcal{P}_2(\mathbb{R}^d) \times \mathbb{R}^{d_0}\to  \mathbb{R}^{d}\times \mathbb{R}^{d\times m},
\end{aligned}
$$
are allowed to be random, and as usual, $\mathcal{P}_2(E)$ denotes the space of probability measures on $E$ having a finite second moment. 
We shall make the following assumptions.\\
\textbf{(A1.1)} The functions $b_0$ and $\sigma_0$  (resp. $b$ and $\sigma$) are $\mathcal{P}^{W^0} \otimes \mathcal{B}(\mathbb{R}^{d_0})\otimes \mathcal{B}(\mathcal{P}(\mathbb{R}^d))$-measurable (resp. $\mathcal{P}^{W^0} \otimes \mathcal{B}(\mathbb{R}^{d})\otimes \mathcal{B}(\mathcal{P}(\mathbb{R}^d)) \otimes \mathcal{B}(\mathbb{R}^{d_0})$-measurable), where $\mathcal{P}^{W^0}$ is the progressive $\sigma$-field associated with the filtration $\mathcal{F}^0_t$ on $[0,T] \times \Omega$ and $\mathcal{B}(\mathcal{P}(\mathbb{R}^d))$ is the Borel $\sigma$-field generated by the metric $W_2$.\\
\textbf{(A1.2)} There exists a constant $K>0$ such that for all $t \in [0,T]$, $\omega \in \Omega$, $x, x' \in \mathbb{R}^d$, $x_0, x'_0 \in \mathbb{R}^{d_0}$ and $\mu, \mu' \in \mathcal{P}_2(\mathbb{R}^d)$,
$$\begin{aligned}
&|(b_0,\sigma_0)(t,\omega, x_0, \mu)-b_0(t, \omega, x'_0, \mu')| \leq K(|x_0-x'_0|+W_2(\mu, \mu')),\\
&|(b,\sigma)(t,\omega,x,\mu,x_0)-b(t, \omega, x', \mu', x'_0)| \leq K(|x-x'|+|x_0-x'_0|+W_2(\mu, \mu')).
\end{aligned}
$$
\textbf{(A1.3)} We have
$$
\mathbb{E}\left[\int^T_0 \vert (b_0,\sigma_0)(t,0,\delta_0) \vert^2+\vert (b,\sigma)(t,0,\delta_0,0) \vert^2 dt\right] < \infty.
$$

\vskip 2pt
Our goal is to study the limiting behaviour of the solution of the system (\ref{fo:finitepoc}) when $N$ tends to infinity. The limit will be given by the so-called \emph{limiting nonlinear processes}, but before defining it, we need to introduce notations and definitions for the regular versions of conditional probabilities
which we use throughout the remainder of the paper.

\subsection{Regular conditional distributions and optional projections}
\label{sub:conditional}
We consider a measurable space $(\Omega,\cF)$ and we assume that $\Omega$ is standard and $\cF$ is its Borel $\sigma$-field to allow us to use regular conditional distributions for any sub-$\sigma$-field of $\mathcal{F}$. In fact, if $(\cG_t)$ is a right continuous filtration, we make use of the existence of a map $\Pi^\cG:[0,\infty)\times\Omega \hookrightarrow \cP(\Omega)$ which is $(\cO,\cB(\cP(\Omega))$-measurable and such that for each $t\ge 0$, $\{\Pi^\cG_t(\omega,A);\,\omega\in\Omega,\,A\in\cF\}$ is a regular version of the conditional probability of $\PP$ given the $\sigma$-field $\cG_t$.
Here $\cO$ denotes the optional $\sigma$-field of the filtration $(\cG_t)$. This result is a direct consequence of Proposition 1 in \cite{Yor} applied to the the process $(X_t)$ given by the identity map of $\Omega$ and the constant filtration $\cF_t\equiv\cF$.
For each $t\ge 0$, we define the probability measures $\mathbb{P}\otimes \Pi^\mathcal{G}_t$ and $ \Pi^\mathcal{G}_t\otimes\mathbb{P}$ on $\Omega^2=\Omega \times \Omega$ via the formulas
\begin{equation}
\label{fo:probas}
\mathbb{P}\otimes \Pi^\mathcal{G}_t(A \times B)=\int_A \Pi^\mathcal{G}_t(\omega, B) \mathbb{P}(d\omega).
\quad\text{and}\quad
\Pi^\mathcal{G}_t\otimes\mathbb{P}(A \times B)=\int_B \Pi^\mathcal{G}_t(\omega, A) \mathbb{P}(d\omega).
\end{equation}
It is easy to check that, integrals of functions of the form $\Omega^2\ni (\omega,\t\omega)\hookrightarrow \varphi(\omega)\psi(\t\omega)$ with respect to these two measures are equal. This shows that these two measures are the same. We will use this result in the following way: if $X$ is measurable and bounded on $\Omega^2$, we can interchange $\omega$ and $\t\omega$ in the integrand of
$$
\int_{\Omega^2}X(\omega,\t\omega)\Pi^\mathcal{G}_t(\omega, d\t\omega) \mathbb{P}(d\omega)
$$
without changing the value of the integral.

In this section, we often use the notation $\mathbb{E}^{\mathcal{G}_t}$ for the expectation with respect to the transition kernel $\Pi_t^\mathcal{G}$, i.e. for all random variable $X : \Omega^2 \ni (\omega,\tilde{\omega})\hookrightarrow X(\omega,\tilde{\omega})\in\mathbb{R}$, we define
$$
\mathbb{E}^{\mathcal{G}_t}[X(\omega,\tilde{\omega})] =\int_\Omega X(\omega, \tilde{\omega}) \Pi^\mathcal{G}_t(\omega, d\tilde{\omega}),
$$
which, as a function of $\omega$, is a random variable on $\Omega$. Also, we still use $\mathbb{E}$ to denote the expectation with respect to the first argument, i.e.
$$
\mathbb{E}[X]=\int_\Omega X(\omega, \tilde{\omega}) \mathbb{P}(d\omega),
$$
which, as a function of $\t\omega$, is a random variable on $\Omega$. Finally, whenever we have a random variable $X$ defined on $\Omega$, we define the random variable $\tilde{X}$ on $\Omega^2$ via the formula
$\tilde{X}(\omega,\tilde{\omega})=X(\tilde{\omega})$.

\subsection{Conditional McKean-Vlasov SDEs}
In order to define properly the limiting nonlinear processes, we first derive a few technical properties of the conditional distribution of a process with respect to a filtration. We now assume that the filtration $(\mathcal{G}_t)$ is a sub-filtration of a right continuous filtration $(\mathcal{F}_t)$, in particular  $\mathcal{G}_t \subseteq \mathcal{F}_t$ for all $t\ge 0$, and that $(X_t)$ is an $\mathcal{F}_t$-adapted continuous process taking values in a Polish space $(E,\mathcal{E})$.  Defining $\mu^X_t(\omega)$ as the distribution of the random variable $X_t$ under the probability measure $\Pi^\cG_t(\omega,\,\cdot\,)$, we obtain the following result which we state as a lemma for future reference.

\begin{lemma}
\label{proj}
There exists a stochastic measure flow $\mu^X:[0,\infty)\times \Omega \rightarrow \mathcal{P}(E)$ such that
\begin{enumerate}
\item $\mu^X$ is $\mathcal{P}$/$\mathcal{B}(\mathcal{P}(E))$-measurable, where $\mathcal{P}$ is the progressive $\sigma$-field associated to $(\mathcal{G}_t)$ on $[0,\infty)\times \Omega$, and $\mathcal{B}(\mathcal{P}(E))$ the Borel $\sigma$-field of the weak topology on $\mathcal{P}(E)$.
\item $\forall t \geq 0$, $\mu^X_t$ is a regular conditional distribution of $X_t$ given $\mathcal{G}_t$;
\end{enumerate}
\end{lemma}

\vskip 4pt\noindent
We first study the well-posedness of the SDE:
\begin{equation}\label{fo:cMkV}
dX_t=b(t,X_t,\mathcal{L}(X_t\vert \mathcal{G}_t))dt+\sigma (t,X_t,\mathcal{L}(X_t\vert \mathcal{G}_t))dW_t.
\end{equation}
We  say that this SDE is of the conditional McKean-Vlasov type because the conditional distribution of $X_t$ with respect to $\mathcal{G}_t$ enters the dynamics. Note that when $\mathcal{G}_t$ is the trivial $\sigma$-field, (\ref{fo:cMkV}) reduces to a classical McKean-Vlasov SDE.
In the following, when writing $\mathcal{L}(X_t \vert \mathcal{G}_t)$ we always mean $\mu^X_t$, for the stochastic flow $\mu^X$ whose existence is given in Lemma \ref{proj}.

The analysis of the SDE (\ref{fo:cMkV}) is done under the following assumptions. We let $W$ be a $m$-dimensional Wiener process on a probability space $(\Omega,\mathcal{F},\mathbb{P})$, $\mathcal{F}^0_t$ its (raw) filtration, $\mathcal{F}_t=\mathcal{F}_t^W$ its usual $\mathbb{P}$-augmentation, and $\mathcal{G}_t$ a sub-filtration of $\mathcal{F}_t$ also satisfying the usual conditions. We impose the following standard assumptions on $b$ and $\sigma$:

\noindent \textbf{(B1.1)} The function
 $$
(b,\sigma):[0,T]\times \Omega \times \mathbb{R}^n \times \mathcal{P}(\mathbb{R}^n) \ni (t,\omega,x,\mu) \hookrightarrow (b(t,\omega,x,\mu),\sigma(t,\omega,x,\mu)) \in \mathbb{R}^n\times \RR^{n \times m}
$$
is $\mathcal{P}^{\mathcal{G}} \otimes \mathcal{B}(\mathbb{R}^{n})\otimes \mathcal{B}(\mathcal{P}(\mathbb{R}^n))$-measurable, where $\mathcal{P}^{\mathcal{G}}$ is the progressive $\sigma$-field associated with the filtration $\mathcal{G}_t$ on $[0,T] \times \Omega$;\\
\textbf{(B1.2)} There  exists $K>0$ such that for all $t \in [0,T]$, $\omega \in \Omega$, $x, x' \in \mathbb{R}^n$,  and $\mu, \mu' \in \mathcal{P}_2(\mathbb{R}^n)$, we have:
$$
|b(t,\omega, x, \mu)-b(t, \omega, x', \mu')|+|\sigma(t,\omega, x, \mu)-\sigma(t, \omega, x', \mu')| \leq K(|x-x'|+W_2(\mu, \mu')).
$$
\textbf{(B1.3)} It holds:
$$\mathbb{E}\left[\int^T_0 \vert b(t,0,\delta_0) \vert^2+\vert \sigma(t,0,\delta_0) \vert^2 dt\right] < \infty.$$

\begin{definition}
By a (strong) solution of (\ref{fo:cMkV}) we mean an $\mathcal{F}_t$-adapted continuous process $X$ taking values in $\mathbb{R}^n$ such that for all $t \in [0,T]$,
$$X_t=x_0 +\int^t_0 b(s,X_s,\mathcal{L}(X_s\vert \mathcal{G}_s))ds+\int^t_0 \sigma (s,X_s,\mathcal{L}(X_s\vert \mathcal{G}_s))dW_s, \quad \text{a.s..}$$
\end{definition}

In order to establish the well-posedness of (\ref{fo:cMkV}) we need some form of control on the 2-Wasserstein distance between two conditional distributions. We shall use the following dual representation:

\begin{proposition}
If $\mu,\nu\in\mathcal{P}_2(E)$ where $E$ is an Euclidean space, then:
$$
W^2_2(\mu,\nu)=\sup_{\phi \in \mathcal{C}^{\text{Lip}}_b(E)}\bigg(\int_E \phi^* d\mu - \int_E \phi d\nu\bigg),
$$
where $\phi^*(x) := \inf_{z \in E} \phi(z)+|x-z|^2. $
\end{proposition}
\noindent
We shall use the following consequences of this representation.

\begin{lemma}
 \label{lwasserstein}
If $X$ and $Y$ are two random variables of order $2$ taking values in a Euclidean space, and $\mathcal{G}$ a sub-$\sigma$-field of $\mathcal{F}$, then for all $p \geq 2$ we have:
$$
W^p_2(\mathcal{L}(X|\mathcal{G}),\mathcal{L}(Y|\mathcal{G}))\leq \mathbb{E}[|X-Y|^p\vert \mathcal{G}], \text{a.s.}.
$$
By taking expectations on both sides we further have
$$\mathbb{E}\left[W^p_2(\mathcal{L}(X|\mathcal{G}),\mathcal{L}(Y|\mathcal{G}))\right] \leq \mathbb{E}[|X-Y|^p].$$
\end{lemma}
\begin{proof}
By using the above dual representation formula and the characteristic equation for conditional distributions, we get
$$W^2_2(\mathcal{L}(X|\mathcal{G}),\mathcal{L}(Y|\mathcal{G}))=\sup_{\phi \in \mathcal{C}^{\text{Lip}}_b(E)}\mathbb{E}[\phi^*(X)-\phi(Y)|\mathcal{G}] \leq \mathbb{E}[|X-Y|^2\vert \mathcal{G}],
$$
and the first inequality follows by applying the conditional Jensen's inequality.
\end{proof}

We then have the following well-posedness result.

\begin{proposition}\label{pMkV}
The conditional McKean-Vlasov SDE (\ref{fo:cMkV}) has a unique strong solution. Moreover, for all $p \geq 2$, if we replace the assumption (B1.3) by
$$\mathbb{E}\int^T_0 \vert b(t,0,\delta_0)\vert^p+\vert \sigma(t,0,\delta_0)\vert^p dt < \infty,$$
then, the solution of (\ref{fo:cMkV}) satisfies
$$\mathbb{E}\left[\sup_{0 \leq t \leq T} \vert X_t\vert^p \right] < \infty.$$
\end{proposition}

\begin{proof}
The proof is an application of the contraction mapping theorem. For each $c>0$, we consider the space of all $\mathcal{F}_t$-progressively measurable processes satisfying
$$
\|X\|_c^2:=\mathbb{E}\left[\int^T_0 e^{-ct}|X_t|^2 dt\right]<\infty.
$$
This space will be denoted by $\mathbb{H}^2_c$. It can be easily proven to be a Banach space. Furthermore, for all $X \in \mathbb{H}^2_c$, we have 
$$
\mathcal{L}(X_t \vert \mathcal{G}_t) \in \mathcal{P}_2(\mathbb{R}^n), \quad \text{a.s., a.e..}
$$
and we can define
$$
U_t=x_0 +\int^t_0 b(s,X_s,\mathcal{L}(X_s\vert \mathcal{G}_s))ds+\int^t_0 \sigma(s,X_s,\mathcal{L}(X_s\vert \mathcal{G}_s))dW_s.
$$
It is easy to show that $U \in \mathbb{H}^2_c$. On the other hand, if we fix $X,X' \in \mathbb{H}^2_c$ and let $U$ and $U'$ be the processes defined via the above equality from $X$ and $X'$ respectively, we have
$$
\begin{aligned}
&\mathbb{E}\left[\left|\int^t_0  b(s,X'_s,\mathcal{L}(X'_s\vert \mathcal{G}_s))-b(s,X_s,\mathcal{L}(X_s\vert \mathcal{G}_s)) ds\right|^2\right]\\
 &\phantom{????}\leq 2 T K^2 \mathbb{E}\left[\int^t_0 \vert X'_s-X_s \vert^2+W^2_2(\mathcal{L}(X'_s\vert \mathcal{G}_s),\mathcal{L}(X_s\vert \mathcal{G}_s))ds\right]\\
&\phantom{????}\leq 2 T K^2 \mathbb{E}\left[\int^t_0 \vert X'_s-X_s\vert^2 ds\right],
\end{aligned}
$$
and we have the same type of estimate for the stochastic integral term by replacing the Cauchy-Schwarz inequality by the Ito isometry. This yields
$$\begin{aligned}
\Vert U'-U \Vert^2_c=&\mathbb{E}\left[\int^T_0 e^{-ct} \vert U'_t-U_t \vert^2dt\right] \\
\leq & 2 (T+1) K^2\mathbb{E}\left[\int^T_0 e^{-ct} \left(\int^t_0 \vert X'_s-X_s \vert^2 ds\right)dt\right]\\
\leq &\frac{2 (T+1) K^2}{c}\Vert X'-X \Vert^2_c,
\end{aligned}$$
and this proves that the map $X\to U$ is a strict contraction in the Banach space $\mathbb{H}^2_c$ if we choose $c$ sufficiently large. The fact that the solution possesses finite moments can be obtained by using standard estimates and Lemma ~\ref{lwasserstein}. We omit the proof here.
\end{proof}

In the above discussion, $\mathcal{G}_t$ is a rather general sub-filtration of the Brownian filtration $\mathcal{F}^{W}_t$. From now on, we shall restrict ourselves to sub-filtrations $\mathcal{G}_t$ equal to the Brownian filtration generated by the first $r$ components of $W$ for some $r < m$. We rewrite (\ref{fo:cMkV}) as
\begin{equation}\label{fo:cMkV2}
dX_t=b(t,X_t,\mathcal{L}(X_t\vert \mathcal{G}^W_t))dt+\sigma (t,X_t,\mathcal{L}(X_t\vert \mathcal{G}^W_t))dW_t,
\end{equation}
and we expect that the solution of the SDE (\ref{fo:cMkV2}) is given by a deterministic functional of the Brownian paths. In order to prove this fact in a rigorous way, we need the following notion.

\begin{definition}
By a set-up we mean a 4-tuple $(\Omega, \mathcal{F}, \mathbb{P}, W)$ where $(\Omega, \mathcal{F}, \mathbb{P})$ is a probability space with a $d$-dimensional Wiener process $W$. We use $\mathcal{F}^W_t$ to denote the natural filtration generated by $W$ and $\mathcal{G}^W_t$ to denote the natural filtration generated by the first $r$ components of $W$. By the canonical set-up we mean $(\Omega^c, \mathcal{F}^c, \mathbb{W}, B)$, where $\Omega^c=C([0,T];\mathbb{R}^m)$, $\mathcal{F}^c$ is the Borel $\sigma$-field associated with the uniform topology, $\mathbb{W}$ is the Wiener measure and $B_t$ is the coordinate (marginal) projection.
\end{definition}

Proposition \ref{pMkV} basically states that the SDE (\ref{fo:cMkV2}) is uniquely solvable on any set-up, and in particular it is uniquely solvable on the canonical set-up. The solution in the canonical set-up, denoted by $X^c$, gives us a measurable functional from $C([0,T];\mathbb{R}^d)$ to $C([0,T];\mathbb{R}^n)$. Because of the important role played by this functional, in the following we use $\Phi$ (instead of $X^c$) to denote it.

\begin{lemma}\label{ltransform}
Let $\psi:C([0,T];\mathbb{R}^m) \rightarrow \mathbb{R}^n$ be $\mathcal{F}^B_t$-measurable, then we have
$$\mathcal{L}(\psi \vert \mathcal{G}^B_t)(W_\cdot)=\mathcal{L}(\psi(W_\cdot)\vert \mathcal{G}^W_t).$$
\end{lemma}
\begin{proof}
By the definition of conditional distributions, it suffices to prove that for all bounded measurable functions $f:\mathbb{R}^n \rightarrow \mathbb{R}^+$ we have
$$\mathbb{E}\left[f(\psi(W_\cdot))\vert \mathcal{G}^W_t\right]=\mathbb{E}\left[f(\psi)\vert \mathcal{G}^B_t\right](W_\cdot),$$
and by using the definition of conditional expectations the above equality can be easily proved.
\end{proof}

With the help of Lemma \ref{ltransform}, we can state and prove

\begin{proposition}\label{pfunctional}
On any set-up $(\Omega, \mathcal{F}, \mathbb{P}, W)$, the solution of (\ref{fo:cMkV2}) is given by
$$X_\cdot=\Phi(W_\cdot).$$
\end{proposition}
\begin{proof}
We are going to check directly that $\Phi(W_\cdot)$ is a solution of (\ref{fo:cMkV2}). By the definition of $\Phi$ as the solution of (\ref{fo:cMkV2}) on the canonical set-up, we have
$$
\Phi(\mathtt{w})=x_0+\int^t_0 b(s,\Phi(\mathtt{w})_s, \mathcal{L}(\Phi(\cdot)_s \vert \mathcal{G}^B_s)(\mathtt{w}))ds+\int^t_0 \sigma(s,\Phi(\mathtt{w})_s, \mathcal{L}(\Phi(\cdot)_s \vert \mathcal{G}^B_s)(\mathtt{w}))dB_s, \mathbb{W}-\text{a.s.,}
$$
where $\mathtt{w}$ stands for a generic element in the canonical space $C([0,T];\mathbb{R}^m)$. By using Lemma \ref{ltransform} we thus have
$$
\Phi(W_\cdot)=x_0+\int^t_0 b(s,\Phi(W_\cdot)_s, \mathcal{L}(\Phi(W_\cdot)_s \vert \mathcal{G}^B_s))ds+\int^t_0 \sigma(s,\Phi(W_\cdot)_s, \mathcal{L}(\Phi(W_\cdot)_s \vert \mathcal{G}^B_s))dW_s, \mathbb{P}-\text{a.s.,}
$$
which proves the desired result.
\end{proof}

\subsection{The Nonlinear Processes}
The limiting nonlinear processes associated with the particle system (\ref{fo:finitepoc}) is defined as the solution of
\begin{equation}
\label{fo:Nonlinear}
\begin{cases}
dX^0_t=b_0(t,X^0_t, \mathcal{L}(X^1_t\vert \mathcal{F}^0_t))dt+\sigma_0(t,X^0_t, \mathcal{L}(X^1_t\vert \mathcal{F}^0_t)) dW^0_t,\\
dX^i_t=b(t,X^i_t, \mathcal{L}(X^i_t\vert \mathcal{F}^0_t),X^0_t)dt+\sigma(t,X^i_t, \mathcal{L}(X^i_t\vert \mathcal{F}^0_t),X^0_t) dW^i_t, \quad i \geq 1,\\
X^0_0=x^0_0,\qquad
X^i_0=x_0, \quad i \geq 1.
\end{cases}
\end{equation}
Under the assumptions (A1.1)-(A1.3), the unique solvability of this system is ensured by Proposition \ref{pMkV}. Due to the strong symmetry among the processes $(X^i)_{i \geq 1}$, we first prove the following proposition.

\begin{proposition}
For all $i \geq 1$, the solution of (\ref{fo:Nonlinear}) solves the conditional McKean-Vlasov SDE
$$\begin{cases}
dX^0_t=b_0(t,X^0_t,\mathcal{L}(X^i_t\vert \mathcal{F}^0_t))dt+\sigma_0(t,X^0_t,\mathcal{L}(X^i_t\vert \mathcal{F}^0_t)) dW^0_t,\\
dX^i_t=b(t,X_t,\mathcal{L}(X^i_t\vert\mathcal{F}^0_t),X^0_t)dt+\sigma(t,X^i_t, \mathcal{L}(X^i_t\vert \mathcal{F}^0_t),X^0_t) dW^i_t,
\end{cases}$$
and for all fixed $t \in [0,T]$, the random variables $(X^i_t)_{i \geq 1}$ are $\mathcal{F}^0_t$-conditionally i.i.d..
\end{proposition}
\begin{proof}
This is an immediate consequence of Proposition \ref{pfunctional}.
\end{proof}
Now that the nonlinear processes are well-defined, in the next subsection we prove that these processes give  the limiting behaviour of (\ref{fo:finitepoc}) when $N$ tends to infinity.

\subsection{Conditional Propagation of Chaos}
We extend the result of the unconditional theory to the conditional case involving the influence of a major player. 
As in the classical case, the propagation appears in a strong path wise sense.

\begin{theorem}\label{tpropagation}
There exists a constant $C$ such that
$$\max_{0 \leq i \leq N}\mathbb{E}[\sup _{0 \leq t \leq T}|X^{i,N}_t-X^i_t|^2] \leq CN^{-2/(d+4)},$$
where $C$ only depends on $T$,  the Lipschitz constants of $b_0$ and $b$ and
$$\eta=\mathbb{E}\left[\int^T_0 \vert X^1_t\vert^{d+5}dt\right]$$
\end{theorem}

\begin{proof}
We first note that, by the SDEs satisfied by $X^0$ and $X^{0,N}$ and the Lipschitz conditions on the coefficients,
$$\begin{aligned}
&|X^{0,N}_t-X^0_t|^2\\
= &(\int^t_0 b_0(s,X^{0,N}_s, \frac{1}{N}\sum^N_{j=1}\delta _{X^{j,N}_s})-b(s,X^0_s,\mu_s)ds)^2 \\
\leq &K(\int^t_0 |X^{0,N}_s-X^0_s|^2ds+ \int^t_0 W^2_2(\frac{1}{N}\sum^N_{j=1}\delta _{X^{j,N}_s},\frac{1}{N}\sum^N_{j=1}\delta_{X^j_s})ds\\
&+ \int^t_0 W^2_2(\frac{1}{N}\sum^N_{j=1}\delta_{X^j_s},\mu_s)ds)\\
\leq &K(\int^t_0 |X^{0,N}_s-X^0_s|^2ds+ \int^t_0 \frac{1}{N}\sum^N_{j=1}(X^{j,N}_s-X^j_s)^2ds\\
&+ \int^t_0 W^2_2(\frac{1}{N}\sum^N_{j=1}\delta_{X^j_s},\mu_s)ds).
\end{aligned}$$

We take the supremum and the expectation on both sides, by the exchangeability we get
$$\begin{aligned}
&\mathbb{E}[\sup_{0 \leq s \leq t}|X^{0,N}_s-X^0_s|^2]\\
\leq &K(\int^t_0 \mathbb{E}[\sup_{0 \leq u \leq s}|X^{0,N}_u-X^0_  u|^2]ds+\int^t_0\mathbb{E}[(X^{1,N}_s-X^1_s)^2]ds\\
&+\int^t_0 \mathbb{E}[W^2_2(\frac{1}{N}\sum^N_{j=1}\delta _{X^j_s},\mu_s)]ds)\\
\leq &K(\int^t_0 \mathbb{E}[\sup_{0 \leq u \leq s}|X^{0,N}_u-X^0_u|^2]ds+\int^t_0\mathbb{E}[\sup_{0 \leq u \leq s}|X^{1,N}_u-X^1_u|^2]ds\\
&+\int^t_0 \mathbb{E}[W^2_2(\frac{1}{N}\sum^N_{j=1}\delta _{X^j_s},\mu_s)]ds),
\end{aligned}
$$
By following the above computation we can readily obtain the same type of estimate for $X^{1,N}-X^1$:
$$\begin{aligned}
&\mathbb{E}[\sup_{0 \leq s \leq t}|X^{1,N}_s-X^1_s|^2]
\leq K'(\int^t_0 \mathbb{E}[\sup_{0 \leq u \leq s}|X^{0,N}_u-X^0_u|^2]ds+\int^t_0\mathbb{E}[\sup_{0 \leq u \leq s}|X^{1,N}_u-X^1_u|^2]ds\\
&\phantom{?????????????????????????????}+\int^t_0 \mathbb{E}[W^2_2(\frac{1}{N}\sum^N_{j=1}\delta _{X^j_s},\mu_s)]ds),
\end{aligned}$$
by summing up the above two inequality and using the Gronwall's inequality we get
$$\begin{aligned}
&\mathbb{E}[\sup_{0 \leq t \leq T}|X^{0,N}_t-X^0_t|^2]+\mathbb{E}[\sup_{0 \leq t \leq T}|X^{1,N}_t-X^1_t|^2]\\
\leq &K\int^T_0 \mathbb{E}[W^2_2(\frac{1}{N}\sum^N_{j=1}\delta _{X^j_t},\mu_t)]ds \leq K\mathbb{E}\left[\int^T_0 \vert X^1_t\vert^{d+5}\right] N^{-2/(d+4)},
\end{aligned}$$
where the second inequality comes from a direct application of Lemma ~\ref{le:RR}, with the help of Lemma ~\ref{lwasserstein}, and this proves the desired result.
\end{proof}

\section{Appendix: A Maximum Principle for Conditional McKean-Vlasov Control Problems}
\label{se:Pontryagin}
In this last section, we establish a version of the sufficient part of the stochastic Pontryagin maximum principle for a type of conditional McKean-Vlasov control problem. In some sense these results are extensions of the results in \cite{CarmonaDelarue_ecp}, and we will refer the reader to \cite{CarmonaDelarue_ecp} for details and proofs. The setup is the following: $(\Omega, \mathcal{F},\mathbb{P})$ is a probability space, $(\mathcal{F}_t)$ is a filtration satisfying the usual conditions defined on $\Omega$. $(\mathcal{G}_t)$ and $(\mathcal{H}_t)$ are two sub-filtrations of $(\mathcal{F}_t)$ also satisfying the usual conditions, and $(W_t)$ is a $n$-dimensional $\mathcal{F}_t$-Wiener process. We assume that the probability space $\Omega$ is standard. \\
\indent The controlled dynamics are given by
\begin{equation}
\label{fo:SDEB}
dX_t=b(t,X_t,\mathcal{L}(X_t \vert \mathcal{G}_t),u_t)dt+\sigma(t,X_t,\mathcal{L}(X_t \vert \mathcal{G}_t),u_t)dW_t,
\end{equation}
with initial condition $X_0 = x_0$, and the objective function to minimize is given by
$$
J(u)=\mathbb{E}\left[\int^T_0 f(t,X_t,\mathcal{L}(X_t \vert \mathcal{G}_t),u_t)+g(X_T,\mathcal{L}(X_T \vert \mathcal{G}_T))\right],
$$
where $X$ is $d$-dimensional and $u$ takes values in $U \subset \mathbb{R}^k$ which is convex. The set of admissible controls is the space $\HH^{2,k}(\cH_t,U)$ defined in \eqref{fo:H2d}. We shall use the assumptions:

\noindent (\textbf{A2.1}) For all $x\in \mathbb{R}^d$, $\mu \in \mathcal{P}_2(\mathbb{R}^d)$ and $u \in U$, $[0,T] \ni t \hookrightarrow (b,\sigma) (t,x,\mu,u)$
is square-integrable.\\
\noindent (\textbf{A2.2}) For all $t \in [0,T]$, $x,x' \in \mathbb{R}^d$, $\mu,\mu' \in \mathcal{P}_2(\mathbb{R}^d)$ and $u \in U$, we have
$$\vert b(t,x',\mu',u)-b(t,x,\mu,u)\vert+\vert \sigma(t,x',\mu',u)-\sigma(t,x,\mu,u)\vert \leq c(\vert x'-x\vert+W_2(\mu',\mu)).$$
\noindent (\textbf{A2.3}) The coefficient functions $b$, $\sigma$, $f$ and $g$ are differentiable with respect to $x$ and $\mu$.\\

We note that under (A2.1-2), for all admissible controls the controlled SDE (\ref{fo:SDEB}) has a unique solution which is square-integrable. (A2.3) will be used in defining the adjoint processes.

\subsection{Hamiltonian and Adjoint Processes}
The Hamiltonian of the problem is defined as
$$
H(t,x,\mu,y,z,u)=\langle y,b(t,x,\mu,u)\rangle+\langle z, \sigma(t,x,\mu,u)\rangle+f(t,x,\mu,u).
$$
Given an admissible control $u\in\HH^{2,k}(\cH_t,U)$, the associated adjoint equation is defined as the following BSDE:
\begin{equation}
\begin{cases}
\begin{aligned}
dY_t=&-\partial_x H(t,X_t,\mathcal{L}(X_t\vert \mathcal{G}_t),Y_t,Z_t,u_t)dt+Z_t dW_t\\
     &-\mathbb{E}^{\mathcal{G}_t}[\partial_\mu H(t,\tilde{X}_t,\mathcal{L}(\tilde{X}_t\vert\mathcal{G}_t),\tilde{Y}_t,\tilde{Z}_t,\tilde{u}_t)(X_t)],
\end{aligned}\\
Y_T=\partial_x g(X_T,\mathcal{L}(X_T \vert \mathcal{G}_T))+\mathbb{E}^{\mathcal{G}_T}[\partial_\mu g(\tilde{X}_T,\mathcal{L}(\tilde{X}_T\vert \mathcal{G}_T))(X_T)],
\end{cases}
\end{equation}
where $X=X^u$ denotes the state controlled by $u$, and whose dynamics are given by \eqref{fo:SDEB}. We refer the reader to \cite{CarmonaDelarue_ecp} for the definition of differentiability with respect to the measure argument. This BSDE is of the McKean-Vlasov type because of the presence of the conditional distributions of various $X_t$ in the coefficients and the terminal condition. However, standard fixed point arguments can be used to prove existence and uniqueness of a solution to these equations.

\subsection{Sufficient Pontryagin Maximum Principle}
The following theorem gives us a sufficient condition of optimality.
\begin{theorem}
On the top of assumptions (A2.1-3), we assume that\\
\noindent (1) The function $\mathbb{R}^d \times \mathcal{P}_2(\mathbb{R}^d) \ni (x,\mu) \hookrightarrow g(x,\mu)$ is convex.\\
\noindent (2) The function $\mathbb{R}^d \times \mathcal{P}_2(\mathbb{R}^d) \times U \ni (x,\mu,u) \hookrightarrow H(t,x,\mu,Y_t,Z_t,u)$ is convex $dt \otimes \mathbb{P}$ a.e.\\
\noindent (3) For any admissible control $u'$ we have the following integrability condition
\begin{equation}
\begin{aligned}
\mathbb{E}\left[\left(\int^T_0 \Vert \sigma(t,X'_t,\mathcal{L}(X'_t\vert \mathcal{G}_t),u'_t)\cdot Y_t\Vert^2 dt\right)^{\frac{1}{2}}\right]< \infty,\quad \mathbb{E}\left[\left(\int^T_0 \Vert X'_t\cdot Z_t\Vert^2 dt\right)^{\frac{1}{2}}\right]< \infty.
\end{aligned}
\end{equation}
Moreover, if
\begin{equation}\label{fo:BDG}
\mathbb{E}[H(t,X_t,\mathcal{L}(X_t\vert \mathcal{G}_t),Y_t,Z_t,u_t)\vert \mathcal{H}_t]=\inf_{u \in U} \mathbb{E}[H(t,X_t,\mathcal{L}(X_t\vert \mathcal{G}_t),Y_t,Z_t,u)\vert \mathcal{H}_t],
\end{equation}
then $(u_t)_{0 \leq t \leq T}$ is an optimal control of the conditional McKean-Vlasov control problem.
\end{theorem}
\begin{proof}
The various steps of the proof of theoremTheorem 4.6 in \cite{CarmonaDelarue_ecp} can be followed \emph{mutatis mutandis} once we remarks that, the interchanges of variables made in Theorem 4.6 of \cite{CarmonaDelarue_ecp} when using independent copies can be done in the same way in the present situation. Indeed, the justification for these interchanges was given at the end of Subsection \ref{sub:conditional} earlier in the previous section.
\end{proof}
One final observation is that a sufficient condition for the integrability condition (\ref{fo:BDG}) is that, on the top of (A2.1)-(A2.3), we have $Y \in \mathbb{S}^{2,d}$ and $Z \in \mathbb{H}^{2,d\times n}$, which is an easy consequence of the Burkholder-Davis-Gundy inequality.

\end{document}